\documentclass[12pt]{article}
\usepackage{geometry}
\usepackage{bbm}
\usepackage{amsfonts}
\usepackage{pifont}
\usepackage{bbding}
\usepackage{latexsym}
\usepackage{mathrsfs}
\usepackage{amsmath}
\usepackage{cite}
\usepackage{amsthm}
\usepackage{amssymb}
\usepackage{graphicx}

\newtheorem{theorem}{\indent Theorem}[section]
\newtheorem{corollary}{\indent Corollary}[section]
\newtheorem{proposition}{\indent Proposition}[section]
\newtheorem{definition}{\indent Definition}[section]
\newtheorem{lemma}{\indent Lemma}[section]
\newtheorem{remark}{\indent Remark}[section]
\numberwithin{equation}{section}

\date{}
\begin{document}

\begin{center}
{\Large \bf Positive solutions for the fractional  Schr\"{o}dinger equations with logarithmic and critical nonlinearities\footnote{This work is supported by NSF DMS-1804497
and Fundamental Research Grant for Central Universities 2018QNA35.} }\\
\vspace{0.5cm} {Haining Fan}\\ \vspace{6pt}
{\footnotesize {\em{ Department of Mathematics, China University of Mining and Technology, Xuzhou, Jiangsu 221116, China}}}

\vspace{0.2cm} {Zhaosheng Feng\footnote{Corresponding author: zhaosheng.feng@utrgv.edu; fax: (956) 665-5091.}}\\  \vspace{6pt}
{\footnotesize{\em{ School of Mathematical and Statistical Sciences, University of Texas,  Edinburg, Texas 78539, USA}}}

\vspace{0.2cm} {Xingjie Yan
}\\
\vspace{6pt}
{\footnotesize{\em{
Department of Mathematics, China University of Mining and Technology, Xuzhou, Jiangsu 221116, China
}}}

\end{center}

\begin{abstract}
In this paper, we study  a class of fractional  Schr\"{o}dinger equations involving logarithmic and critical nonlinearities on an unbounded domain, and show that such an equation with positive or sign-changing weight potentials  admits at least one positive ground state solution and the associated energy is positive (or negative).  By applying the Nehari manifold method and Ljusternik-Schnirelmann category, we deeply investigate how the weight potential affects the multiplicity of positive solutions, and obtain the relationship between the number of positive solutions and the category of some sets related to the weight potential.


\end{abstract}

\section{Introduction}\noindent

The aim of this article is to study how the weight potential affects the existence of ground state solutions and the number of positive  solutions of the fractional  Schr\"{o}dinger equation:
\begin{equation}\label{e1.1}
(-\Delta)^\alpha u+u=\lambda a(x)u\ln|u|+b(x)|u|^{2_\alpha^*-2}u,\ \, x\in\mathbb{R}^N,
\end{equation}
where $\alpha\in(0,1)$, $\lambda>0$,  $N>4\alpha$, $a(x)$ and $b(x)$ are continuous and bounded weight potentials, and $2_\alpha^*=\frac{2N}{N-2\alpha}$ is the fractional critical Sobolev exponent. Let $\wp(\mathbb{R}^N)$ denote the Schwartz space of rapidly decaying $\mathcal{C}^\infty$ functions in $\mathbb{R}^N$. The operator $(-\Delta)^\alpha$ is the  fractional Laplacian  defined by the Riesz potential \cite{1}:

$$(-\Delta)^{\alpha}u(x)=-\frac{C(N,\alpha)}{2}\int_{\mathbb{R}^N}\frac{u(x+y)+u(x-y)-2u(x)}{|y|^{N+2\alpha}}dy,\ \, x\in\mathbb{R}^N,\ u\in\wp(\mathbb{R}^N)$$
where $$C(N,\alpha)=\left(\int_{\mathbb{R}^N}\frac{1-cos\xi_1}{|\xi|^{N+2\alpha}}d\xi\right)^{-1},\ \, \xi=(\xi_1,\xi_2,...,\xi_N).$$
For the definition of the fractional Laplacian $(-\Delta)^\alpha$ and the fractional Sobolev spaces, we refer the reader to Nezza-Palatucci-Valdinoci \cite{1}.

Recall the classical  Schr\"{o}dinger elliptic equation:
\begin{equation}\label{ss1.2}
-\Delta u+V(x)u= f(x,u)\ \text{in} \ \ \mathbb{R}^N,
\end{equation}
where $f(x,u)$ is a polynomial-type nonlinearity, such as  $f(x,u)=a(x)|u|^{q-2}u+b(x)|u|^{p-2}u$ with $2<q<p\leq2^*=\frac{2N}{N-2}$.
It is well-known that the existence of ground state solutions (least energy solutions) and the number of positive solutions of (\ref{ss1.2}) are affected by the weight potential. For example, when $f(x,u)=\lambda|u|^{q-2}u+|u|^{2^*-2}u$, Brezis-Nirenberg  \cite{5} obtained a positive solution of (\ref{ss1.2}) in a bounded domain for $\lambda\in(0,\lambda_1)$, where $\lambda_1$ is the first eigenvalue of $-\Delta$ with the Dirichlet boundary condition. Later on, when $f(x,u)=a(x)|u|^{q-2}u+|u|^{2^*-2}u$, Brezis \cite{44} studied how the weight potential $a(x)$ affects the number of solutions of (\ref{ss1.2}). For more results related to (\ref{ss1.2}), we refer the reader to \cite{4,3,2} for the subcritical growth and \cite{7,8,6} for the critical case.

Various nonlinearities have a rather diverse group of applications in scientific fields \cite{32,feng1}. For example, logarithmic nonlinearity appears frequently in partial differential equations which are widely applied to quantum mechanics, reaction-diffusion phenomena, nuclear physics,  quantum optics, theory of superfluidity and Bose-Einstein condensation \cite{22}. In particular, for the Schr\"{o}dinger equation with a logarithmic nonlinearity:
\begin{equation}\label{ss1.3}
-\Delta u+V(x)u= a(x)u\ln u^2 \ \ \text{in} \ \mathbb{R}^N,
\end{equation}
where $V(x)$ and $a(x)$ are periodic weight potentials, Squassina-Szulkin \cite{23} studied (\ref{ss1.3}) in $H^1({\Bbb R}^N)$ to establish the existence of infinitely many geometrically distinct solutions. Shuai \cite{39} proved the
existence of positive and sign-changing solutions in $H^1({\Bbb R}^N)$ by using the direction derivative and constrained minimization method. Tanaka-Zhang \cite{29} considered a spatially periodic logarithmic Schr\"{o}dinger equation
and showed that there exist infinitely many multi-bump solutions that are distinct under a $\mathbb{}Z^N$-action. For more results related to (\ref{ss1.3}), we refer the reader to \cite{41,40,43,42} and the references therein.

In recent years, much attention has been focused on studying the problems involving the fractional Laplacian from both mathematical and application points of view \cite{10,14,45,11,12,15, 37,13}.
Laskin \cite{11,12} found a fractional generalization of the Schr\"{o}dinger equation for the wave
function in quantum mechanical systems by considering the L\'{e}vy flights instead of the Brownian motion in the Feynman path integral approaches:
\begin{equation}\label{e1.2}
(-\Delta)^\alpha u+u=f(x,u),~ x\in\mathbb{R}^N.
\end{equation}
By considering different expressions of nonlinearity $f$, quite many profound results have been established on the existence and multiplicity of positive solutions.
For example, motivated by Brezis-Nirenberg  \cite{5}, Servadei-Valdinoci \cite{37} considered the model:
 \begin{equation}\label{f1.1}
\left\{\begin{array}{ll}
(-\Delta)^\alpha u=\lambda u+|u|^{2_\alpha^*-2}u,~& \text{in}~\Omega,\\
u=0,~~&\text{in}~\mathbb{R}^N\backslash\Omega,
\end{array}\right.
\end{equation}
and obtained an extended version of the classical Brezis-Nirenberg result to the case of non-local
fractional operators through variational techniques. For $f(x,u)=a(x)|u|^{q-2}u+|u|^{2_\alpha^*-2}u$ with $0<q<2_\alpha^*-1$ in (\ref{e1.2}), Dipierro-Medina-Peral-Valdinoci \cite{45} presented the existence of solutions by using the Lyapunov-Schmidt reduction method. Moreover, for $0<q<1$, under a new functional setting, a fractional elliptic regularity theory was developed too. For more results related to (\ref{e1.2}), we refer to \cite{18,21,38, 17,20,19,16}. However, most of these results assume that $f$ is of polynomial-type.

There is a natural and interesting question here: if the nonlinearity of the fractional Schr\"{o}dinger equation  contains both logarithmic and critical terms like (\ref{e1.1}),  how about the existence and multiplicity of positive solutions for equation (\ref{e1.1})? This is certainly not a trivial problem, because the logarithmic nonlinearity does not satisfy the monotonicity condition (or Ambrosetti-Rabinowitz condition) and this type of nonlinearity may change sign in $\mathbb{R}^N$. On the other hand, the appearance of logarithmic and critical nonlinearity makes it more difficult for us to prove that the resultant (PS) sequence is convergent, and the nonlocal properties of fractional Laplacian operator also cause great difficulties for multiplicity of positive solutions. To the best of our knowledge,  very little has been undertaken on the fractional Schr\"{o}dinger equations involving both logarithmic and critical nonlinearities.

In the present article, we shall study the existence of ground state solutions of equation (\ref{e1.1}) with positive, or negative, or sign-changing weight potentials, and show how the weight potential affects the number of positive solutions.

Before stating our main results, we introduce some assumptions on $a(x)$ and $b(x)$:
\begin{description}
 \item[$(H_1)$] $\displaystyle\lim_{|x|\rightarrow\infty}a(x)=0$,\, $x\in \mathbb{R}^N$.
 \item[$(H_2)$]There exist a compact set $M=\{z\in\mathbb{R}^N;\, b(z)=\displaystyle\max_{x\in\mathbb{R}^N}b(x)=1\}$ and a positive number $\rho>N$ such that $b(z)-b(x)=O(|x-z|^\rho)$ as $x\rightarrow z$ uniformly in $z\in M$.
  \item[$(H_3)$]  $a(x)>0$,\, $x\in M$.
\end{description}
\begin{remark}\label{r1.1}
Let $M_r=\{x\in \mathbb{R}^N;\, dist(x,M)<r\}$ for $r>0$. Then by $(H_2)-(H_3)$ there exist  $C_0, r_0>0$ such that
\begin{equation*}
  a(x)>0,\ \, x\in M_{r_0}\subset\mathbb{R}^N
\end{equation*}
and
\begin{equation*}
  b(z)-b(x)\leq C_0|x-z|^\rho,\ \,  x\in B_{r_0}(z),
\end{equation*}
uniformly in $z\in M$, where $B_{r_0}(z)=\{x\in \mathbb{R}^N;\, |x-z| <r_0\}$.
\end{remark}
\begin{remark}\label{r1.2}
Define
\begin{equation*}
  b_\infty:=\displaystyle\limsup_{|x|\rightarrow\infty}b(x).
\end{equation*}
Then $ b_\infty<1.$
\end{remark}

\begin{theorem}\label{t1.1}
Assume that condition $(H_1)$ holds and $a(x)$ is negative or sign-changing. Then there exists $\Lambda_1>0$ such that if $\lambda\in(0,\Lambda_1)$, equation (\ref{e1.1}) has a positive ground state solution and the ground energy of (\ref{e1.1}) is negative.
\end{theorem}
\begin{theorem}\label{t1.2}
Assume that conditions $(H_1)-(H_3)$ hold and $a(x)\geq0$. Then there exists $\Lambda_2>0$ such that if $\lambda\in(0,\Lambda_2)$, equation (\ref{e1.1}) has a positive ground state solution and the ground energy of (\ref{e1.1}) is positive.
\end{theorem}

For the  definitions of the ground state solution and ground energy, we will present them in Section 2. The following results are regarding the relationship between the number of positive solutions and the weight potentials $a(x)$ and $b(x)$.

\begin{theorem}\label{t1.3}
Assume that conditions $(H_1)-(H_3)$ hold and  $a(x)$ is  sign-changing. Then for each $\delta<r_0$, there exists $\Lambda_\delta>0$ such that if $\lambda\in(0,\Lambda_\delta)$, equation (\ref{e1.1}) has at least  $cat_{M_\delta}(M)+1$ distinct positive solutions, where $cat$ means the Ljusternik-Schnirelmann category (See \cite{32}).
\end{theorem}

\begin{theorem}\label{t1.4}
Assume that conditions $(H_1)-(H_3)$ hold and  $a(x)\geq0$. Then for each $\delta<r_0$, there exists $\overline{\Lambda}_\delta>0$ such that if $\lambda\in(0,\overline{\Lambda}_\delta)$, equation (\ref{e1.1}) has at least  $cat_{M_\delta}(M)$ distinct positive solutions.
\end{theorem}

Particularly, when $b(x)\equiv1$, we have

\begin{theorem}\label{t1.5}
Assume that condition $(H_1)$ holds and  $a(x)$ is sign-changing. Then there exists $\Lambda_3>0$ such that if $\lambda\in(0,\Lambda_3)$, equation (\ref{e1.1}) has at least two distinct positive solutions.
\end{theorem}

\begin{theorem}\label{t1.6}
Assume that condition $(H_1)$ holds and  $a(x)\geq0$ but $a(x)\not\equiv0$. Then there exists $\overline{\Lambda}_3>0$ such that if $\lambda\in(0,\overline{\Lambda}_3)$, equation (\ref{e1.1}) has at least one positive ground state solution and the ground energy of (\ref{e1.1}) is positive.
\end{theorem}

\begin{remark}
Theorems \ref{t1.5} and \ref{t1.6} are the special cases of Theorems \ref{t1.3} and \ref{t1.4}, respectively, so we will omit the proofs.
\end{remark}

To achieve our aim, the Nehari manifold and the Ljusternik-Schnirelmann theory are main tools in this study.
The main feature which distinguishes this paper from other related works lies in the fact that
in the proofs of our results, one of primary difficulties is that the logarithmic nonlinearity does not satisfy the monotonicity condition or Ambrosetti-Rabinowitz condition and this type of nonlinearity may change sign in $\mathbb{R}^N$, which makes discussions more complicated than those without logarithmic nonlinearity. Another primary obstacle is the lack of compactness caused by the unbounded domain and the critical nonlinearity. Some concentration compactness results for the fractional Schr\"{o}dinger equations seem correct but have not been proved yet and thus cannot be applied directly. All these difficulties mentioned above prevent us from using the classical variational methods in a standard way, so innovative techniques are highly needed.

The remainder of this paper is structured as follows. In Section 2, we recall some basic definitions, present the variational setting for the problem and study some properties of the corresponding Nehari manifold. Moreover, we present the proof of Theorem \ref{t1.1}. In Section 3, we obtain useful estimates  and use them to prove Theorem \ref{t1.2}. Section 4 is dedicated to the proofs of Theorems \ref{t1.3} and \ref{t1.4} by means of the Nehari manifold method and Ljusternik-Schnirelmann category theory.

For convenience of our statements,  throughout this article we will use the following notation.

\begin{itemize}
\item[$\bullet$]$\rightarrow$ (resp. $\rightharpoonup$) the strong (resp. weak) convergence.
\item[$\bullet$]$|\cdot|_r$ the usual norm of the space $L^r(\mathbb{R}^N),\ (1\leq r<\infty)$.
\item[$\bullet$]$|\cdot|_\infty$ denotes the norm of the space $L^\infty(\mathbb{R}^N)$.
\item[$\bullet$]$C$ or $C_i\ (i=0,1,2,\ldots)$ denotes positive constants that may change from line to line.
\item[$\bullet$]$\mathbb{R}^{N+1}_+=\{(x_1,x_2,...,x_{N+1})\in\mathbb{R}^{N+1}|x_{N+1}\geq0\}.$
\item[$\bullet$]$B_r=\{x\in \mathbb{R}^N;|x|<r\}$ denotes a ball of radius $r$ in Euclidean spaces.
\end{itemize}

\section{Preliminaries and Proof of Theorem \ref{t1.1}}\noindent

In this section, we first introduce the definition of $\alpha$-harmonic extension. Then we present the variational setting for the problem and properties of the corresponding Nehari manifold. Finally, we use these properties to prove Theorem \ref{t1.1}.

Denote the fractional Sobolev space $H^\alpha(\mathbb{R}^N)$ as the completion of $C_0^\infty(\mathbb{R}^N)$ with the norm:
\begin{equation*}
  \|u\|_{H^\alpha}:=\left(\displaystyle\int_{\mathbb{R}^N}|(-\Delta)^\frac{\alpha}{2}u|^2dx\right)^{\frac{1}{2}}+|u|_2.
\end{equation*}
Then $H^\alpha(\mathbb{R}^N)\hookrightarrow L^r(\mathbb{R}^N)$, $r\in[2,2_\alpha^*]$, and this embedding is locally compact while $r\in[1,2_\alpha^*)$ (see \cite{1}).

To study the corresponding extension problem, we apply an extension method \cite{33} and define the extension function in $H^\alpha(\mathbb{R}^N)$ as follows.

\begin{definition}\label{d2.1}
Given a function $u\in H^\alpha(\mathbb{R}^N)$,  we define the $\alpha$-harmonic extension $E_\alpha(u)=\omega$ to the problem:
\begin{displaymath}
\left\{\begin{array}{ll}
div(y^{1-2\alpha}\nabla\omega)=0,~ &\text{in}~\mathbb{R}_+^{N+1},\\
\omega=u,~~&\text{on}~\mathbb{R}^N\times\{0\}.
\end{array}\right.
\end{displaymath}
The extension function $\omega(x,y)$ has an explicit expression in term of the Poisson and Riesz kernel, i.e.
\begin{equation*}
\omega(x,y)=P_y^\alpha*u(x)=\displaystyle\int_{\mathbb{R}^N}P_y^\alpha(x-\xi,u)u(\xi)d\xi,
\end{equation*}
where $P_y^\alpha(x)=C(N,\alpha)\frac{y^{2\alpha}}{(|x|^2+y^2)^\frac{N+2\alpha}{2}}$ with a constant $C(N,\alpha)$ such that $\displaystyle\int_{\mathbb{R}^N}P_1^\alpha(x)dx=1$ (see \cite{33}).
 \end{definition}

Define the space
\begin{equation*}
X^\alpha(\mathbb{R}_+^{N+1}):=\left\{\omega(x,y)\in C_0^\infty(\mathbb{R}_+^{N+1});\displaystyle\int_{\mathbb{R}_+^{N+1}}k_\alpha y^{1-2\alpha}|\nabla\omega|^2dxdy+\displaystyle\int_{\mathbb{R}^N}|\omega(x,0)|^2dx<\infty\right\},
\end{equation*}
equipped with the norm:
\begin{equation*}
 \|\omega\|_X=\left(\displaystyle\int_{\mathbb{R}_+^{N+1}} y^{1-2\alpha}|\nabla\omega|^2dxdy+\displaystyle\int_{\mathbb{R}^N}|\omega(x,0)|^2dx\right)^{\frac{1}{2}}.
\end{equation*}

Note that
\begin{equation}\label{e2.1}
 \displaystyle\int_{\mathbb{R}_+^{N+1}}k_\alpha y^{1-2\alpha}|\nabla\omega|^2dxdy=\displaystyle\int_{\mathbb{R}^N}|\Delta^\frac{\alpha}{2}u|^2dx,
\end{equation}
where $\omega=E_\alpha(u)$ and $k_\alpha$ is a normal positive constant \cite{33}. So the function $E_\alpha(\cdot)$ is an isometry between $H^\alpha(\mathbb{R}^N)$ and $X^\alpha(\mathbb{R}_+^{N+1})$.
Then  we  re-formulate (\ref{e1.1})  as follows:
\begin{equation}\label{e2.2}
\left\{\begin{array}{ll}
div(y^{1-2\alpha}\nabla\omega)=0,~ &\text{in}~\mathbb{R}_+^{N+1},\\
-k_\alpha\frac{\partial\omega}{\partial\nu}=-\omega+\lambda a(x)\omega\ln|\omega|+b(x)|\omega|^{2_\alpha^*-2}\omega,~~&\text{on}~\mathbb{R}^N\times\{0\},
\end{array}\right.
\end{equation}
where
\begin{equation*}
  -k_\alpha\frac{\partial\omega}{\partial\nu}=-k_\alpha\displaystyle\lim_{y\rightarrow0^+}y^{1-2\alpha}\frac{\partial\omega}{\partial y}(x,y)= (-\Delta)^\alpha u(x).
\end{equation*}

In what follows, we set $k_\alpha=1$, for simplicity.
If $\omega$ is a solution of (\ref{e2.2}), then the trace $u=tr(\omega)=\omega(x,0)$ is a solution of (\ref{e1.1}). Conversely, it is also true.

Let us introduce some properties on the spaces $X^\alpha(\mathbb{R}_+^{N+1})$ and $L^r(\mathbb{R}^N)$.

\begin{proposition}\cite{17}\label{p2.1}
The embedding $X^\alpha(\mathbb{R}_+^{N+1})\hookrightarrow L^r(\mathbb{R}^N)$ is continuous for $r\in[2,2_\alpha^*]$ and locally compact for $r\in [1,2_\alpha^*)$.
\end{proposition}
\begin{proposition}\cite{17}\label{p2.2}
For every $\omega\in X^\alpha(\mathbb{R}_+^{N+1})$, there holds
\begin{equation*}
S\left(\displaystyle\int_{\mathbb{R}^N}|u(x)|^\frac{2N}{N-2\alpha}dx\right)^\frac{N-2\alpha}{N}\leq\displaystyle\int_{\mathbb{R}_+^{N+1}}y^{1-2\alpha}|\nabla\omega|^2dxdy,
\end{equation*}
where $u=tr(\omega)$. The best constant is given by
 \begin{equation*}
  S=\frac{2\pi^\alpha\Gamma(\frac{2-2\alpha}{2})\Gamma(\frac{N+2\alpha}{2})(\Gamma(\frac{N}{2}))^\frac{2\alpha}{N}}{\Gamma(\alpha)\Gamma(\frac{N-2\alpha}{2})(\Gamma(N))^\frac{2\alpha}{N}},
\end{equation*}
and it is attained when $u=\omega(x,0)$ takes the form:
\begin{equation*}
u_\varepsilon(x)=\frac{C\varepsilon^\frac{N-2\alpha}{2}}{(\varepsilon^2+|x|^2)^\frac{N-2\alpha}{2}}
\end{equation*}
for  an arbitrary $\varepsilon>0$, $\omega_\varepsilon=E_\alpha(u_\varepsilon)$ and $$\displaystyle\int_{\mathbb{R}_+^{N+1}}y^{1-2\alpha}|\nabla\omega_\varepsilon|^2dxdy
=\displaystyle\int_{\mathbb{R}^N}|\omega_\varepsilon(x,0)|^\frac{2N}{N-2\alpha}dx=S^\frac{N}{2\alpha}.$$
\end{proposition}

The following property is concerning the logarithmic Sobolev inequality in the fractional Sobolev space.

\begin{proposition}\cite{34}\label{p2.3}
Let $f\in H^\alpha(\mathbb{R}^N)$ and $\sigma>0$ be any number. Then
  \begin{equation*}
  \displaystyle\int_{\mathbb{R}^N}|f|^2\ln\frac{|f|^2}{|f|_2^2}dx\leq\frac{\sigma^2}{\pi^\alpha}|(-\Delta)^\frac{\alpha}{2}f|_2^2-\left[N+\frac{N}{\alpha}\ln\sigma+\ln\frac{\alpha\Gamma(\frac{N}{2})}{\Gamma(\frac{N}{2\alpha})}\right]|f|_2^2.
  \end{equation*}
\end{proposition}
\begin{remark}\label{r2.1}
From (\ref{e2.1}) and definition \ref{d2.1}, we have
  \begin{align}
  \displaystyle\int_{\mathbb{R}^N}|\omega(x,0)|^2\ln\frac{|\omega(x,0)|^2}{|\omega(x,0)|_2^2}dx &\leq\frac{\sigma^2}{\pi^\alpha}\displaystyle\int_{\mathbb{R}_+^{N+1}}y^{1-2\alpha}|\nabla\omega|^2dxdy\notag
  \\&~~~-\left[N+\frac{N}{\alpha}\ln\sigma+\ln\frac{\alpha\Gamma(\frac{N}{2})}{\Gamma(\frac{N}{2\alpha})}\right]|\omega(x,0)|_2^2 \notag
  \end{align}
  for any $\omega\in X^\alpha(\mathbb{R}_+^{N+1})$. Furthermore, there holds
  \begin{align}\label{e2.6}
 & \displaystyle\int_{\mathbb{R}^N}a(x)|\omega(x,0)|^2\ln|\omega(x,0)|dx\nonumber\\
 =& \frac{1}{2}\displaystyle\int_{\mathbb{R}^N}a(x)|\omega(x,0)|^2\ln|\omega(x,0)|^2dx\notag
 \\ =& \frac{1}{2}\displaystyle\int_{\mathbb{R}^N}a(x)|\omega(x,0)|^2\ln\frac{|\omega(x,0)|^2}{|\omega(x,0)|_2^2}dx+\frac{1}{2}\displaystyle\int_{\mathbb{R}^N}a(x)|\omega(x,0)|^2dx\ln|\omega(x,0)|_2^2\notag \\ \leq& \frac{1}{2}|a|_\infty\frac{\sigma^2}{\pi^\alpha}\displaystyle\int_{\mathbb{R}_+^{N+1}}y^{1-2\alpha}|\nabla\omega|^2dxdy+\frac{1}{2}|a|_\infty\left|N+\frac{N}{\alpha}\ln\sigma+\ln\frac{\alpha\Gamma(\frac{N}{2})}{\Gamma(\frac{N}{2\alpha})}\right||\omega(x,0)|_2^2\notag
\\ &+\frac{1}{2}\displaystyle\int_{\mathbb{R}^N}a(x)|\omega(x,0)|^2dx\ln|\omega(x,0)|_2^2.
  \end{align}
\end{remark}

To analyze (\ref{e2.2}),  we define the associated energy functional by
 \begin{align}
 I_\lambda(\omega):&=\frac{1}{2}\|\omega\|_X^2-\frac{\lambda}{2}\displaystyle\int_{\mathbb{R}^N}a(x)|\omega(x,0)|^2\ln|\omega(x,0)|dx\notag
 \\&~~~~+\frac{\lambda}{4}\displaystyle\int_{\mathbb{R}^N}a(x)|\omega(x,0)|^2dx-\frac{1}{2_\alpha^*}\displaystyle\int_{\mathbb{R}^N}b(x)|\omega(x,0)|^{2_\alpha^*}dx,\notag
\end{align}
where $\omega\in X^\alpha(\mathbb{R}_+^{N+1})$. Then $I_\lambda$ is Fr\'{e}chet differentiable and
 \begin{align}
 I'_\lambda(\omega) \varphi& = \displaystyle\int_{\mathbb{R}_+^{N+1}}y^{1-2\alpha}\nabla\omega(x,y)\nabla\varphi(x,y) dxdy+\displaystyle\int_{\mathbb{R}^N}\omega(x,0)\varphi(x,0)dx\notag
 \\&~~~~-\lambda\displaystyle\int_{\mathbb{R}^N}a(x)\omega(x,0)\varphi(x,0)\ln|\omega(x,0)|dx
 -\displaystyle\int_{\mathbb{R}^N}b(x)|\omega(x,0)|^{2_\alpha^*-2}\omega(x,0)\varphi(x,0)dx \notag
  \end{align}
for any $\varphi\in X^\alpha(\mathbb{R}_+^{N+1})$. It is notable that finding the weak solution of (\ref{e2.2}) is equivalent to finding the critical point of the energy functional $I_\lambda$.

Define
\begin{equation*}
\Phi:=\{\text{nontrivial~weak~solutions~of~(\ref{e2.2})}\}.
\end{equation*}
From (\ref{e2.1}) and Definition \ref{d2.1},  we define the ground energy of equation (\ref{e1.1}) by
\begin{equation*}
d:=\displaystyle\inf_{\omega\in\Phi}I_\lambda(\omega).
\end{equation*}
If $\omega$ is a nontrivial solution of system (\ref{e2.2}) such that $I_\lambda(\omega)=d$, we call that $u:=\omega(x,0)$ is a ground state solution of equation (\ref{e1.1}).

Since $I_\lambda$ is not bounded from below on $X^\alpha(\mathbb{R}_+^{N+1})$, we consider  $I_\lambda$ strictly on the Nehari manifold:
 \begin{equation*}
 N_\lambda:=\{\omega\in X^\alpha(\mathbb{R}_+^{N+1})\backslash\{0\};\, I'_\lambda(\omega)\omega=0\}.
\end{equation*}
 Then $\omega\in N_\lambda$ if and only if
 \begin{equation}\label{e2.8}
 \|\omega\|_X^2-\lambda\displaystyle\int_{\mathbb{R}^N}a(x)|\omega(x,0)|^2\ln|\omega(x,0)|dx-\displaystyle\int_{\mathbb{R}^N}b(x)|\omega(x,0)|^{2_\alpha^*}dx=0.
\end{equation}

We analyze $ N_\lambda$ in terms of the stationary points of fibrering maps \cite{4} that $\phi_\omega:\mathbb{R}^+\rightarrow \mathbb{R}$ is defined by
\begin{equation*}
 \phi_\omega(t):=I_\lambda(t\omega).
\end{equation*}
Then we have
 \begin{align}\label{eee2.9}
 &\phi_\omega(t)=\frac{t^2}{2}\|\omega\|_X^2-\frac{t^{2_\alpha^*}}{2_\alpha^*}\displaystyle\int_{\mathbb{R}^N}b(x)|\omega(x,0)|^{2_\alpha^*}dx
\\&-\lambda\frac{t^2}{2}\left(\displaystyle\int_{\mathbb{R}^N}a(x)|\omega(x,0)|^2\ln|\omega(x,0)|dx+\ln t\displaystyle\int_{\mathbb{R}^N}a(x)|\omega(x,0)|^2dx-\frac{1}{2}\displaystyle\int_{\mathbb{R}^N}a(x)|\omega(x,0)|^2dx\right),\notag
\end{align}
\begin{align}
 \phi'_\omega(t)=&t\|\omega\|_X^2-t^{2_\alpha^*-1}\displaystyle\int_{\mathbb{R}^N}b(x)|\omega(x,0)|^{2_\alpha^*}dx\notag
\\&-\lambda t\left(\displaystyle\int_{\mathbb{R}^N}a(x)|\omega(x,0)|^2\ln|\omega(x,0)|dx+\ln t\displaystyle\int_{\mathbb{R}^N}a(x)|\omega(x,0)|^2dx\right)\notag
\end{align}
and
\begin{align}
 &\phi''_\omega(t)=\|\omega\|_X^2-(2_\alpha^*-1)t^{2_\alpha^*-2}\displaystyle\int_{\mathbb{R}^N}b(x)|\omega(x,0)|^{2_\alpha^*}dx\notag
\\&-\lambda \left(\displaystyle\int_{\mathbb{R}^N}a(x)|\omega(x,0)|^2\ln|\omega(x,0)|dx+\ln t\displaystyle\int_{\mathbb{R}^N}a(x)|\omega(x,0)|^2dx+\displaystyle\int_{\mathbb{R}^N}a(x)|\omega(x,0)|^2dx\right).\notag
\end{align}
It is easy to see that $\omega\in N_\lambda$ if and only if $\phi'_\omega(1)=0.$

We split $N_\lambda$ into three subsets $N^+_\lambda$, $N^-_\lambda$, and $N^0_\lambda$ that correspond to local minima, local maxima, and points of inflection of fibrering maps respectively, i.e.
\begin{align}
&N^+_\lambda:=\{\omega\in N_\lambda;\phi''_\omega(1)>0\}=\{t\omega\in X^\alpha(\mathbb{R}_+^{N+1})\backslash\{0\};\, \phi'_\omega(t)=0,\phi''_\omega(t)>0\},\notag
\\&N^-_\lambda:=\{\omega\in N_\lambda;\phi''_\omega(1)<0\}=\{t\omega\in X^\alpha(\mathbb{R}_+^{N+1})\backslash\{0\};\, \phi'_\omega(t)=0,\phi''_\omega(t)<0\},\notag
\\&N^0_\lambda:=\{\omega\in N_\lambda;\phi''_\omega(1)=0\}=\{t\omega\in X^\alpha(\mathbb{R}_+^{N+1})\backslash\{0\};\, \phi'_\omega(t)=0,\phi''_\omega(t)=0\}.\notag
\end{align}
Note that if $\omega\in N_\lambda$, then $\phi''_\omega(1)=-\lambda\displaystyle\int_{\mathbb{R}^N}a(x)|\omega(x,0)|^2dx-(2_\alpha^*-2)\displaystyle\int_{\mathbb{R}^N}b(x)|\omega(x,0)|^{2_\alpha^*}dx$.
Thus, we get the equivalent expressions:
\begin{align}
&N^+_\lambda:=\{\omega\in N_\lambda;\, \lambda\displaystyle\int_{\mathbb{R}^N}a(x)|\omega(x,0)|^2dx+(2_\alpha^*-2)\displaystyle\int_{\mathbb{R}^N}b(x)|\omega(x,0)|^{2_\alpha^*}dx<0\},\notag
\\&N^-_\lambda:=\{\omega\in N_\lambda;\, \lambda\displaystyle\int_{\mathbb{R}^N}a(x)|\omega(x,0)|^2dx+(2_\alpha^*-2)\displaystyle\int_{\mathbb{R}^N}b(x)|\omega(x,0)|^{2_\alpha^*}dx>0\},\notag
\\&N^0_\lambda:=\{\omega\in N_\lambda;\, \lambda\displaystyle\int_{\mathbb{R}^N}a(x)|\omega(x,0)|^2dx+(2_\alpha^*-2)\displaystyle\int_{\mathbb{R}^N}b(x)|\omega(x,0)|^{2_\alpha^*}dx=0\}.\notag
\end{align}

\begin{lemma}\label{l2.0}
If $\displaystyle\int_{\mathbb{R}^N}a(x)|\omega(x,0)|^2dx\leq0$, then we have either
\begin{equation}\label{s2.1}
\|\omega\|_X\leq1,
 \end{equation}
 or
 \begin{equation*}
\displaystyle\int_{\mathbb{R}^N}a(x)|\omega(x,0)|^2\ln|\omega(x,0)|dx\leq C\|\omega\|_X^2
 \end{equation*}
 for some $C>0$ independent of $\omega\in X^\alpha \left(\mathbb{R}_+^{N+1}\right)$.
\end{lemma}
\begin{proof}
To estimate $\displaystyle\int_{\mathbb{R}^N}a(x)|\omega(x,0)|^2\ln|\omega(x,0)|dx$, we re-write it as
\begin{align*}\label{s2.13}
&\displaystyle\int_{\mathbb{R}^N}a(x)|\omega(x,0)|^2\ln|\omega(x,0)|dx\notag
\\=& \displaystyle\int_{\mathbb{R}^N}a(x)|\omega(x,0)|^2\ln\frac{|\omega(x,0)|}{\|\omega\|_X}dx+\ln\|\omega\|_X\displaystyle\int_{\mathbb{R}^N}a(x)|\omega(x,0)|^2dx\notag
\\=& I_1+I_2,
\end{align*}
where $I_1:=\displaystyle\int_{\mathbb{R}^N}a(x)|\omega(x,0)|^2\ln\frac{|\omega(x,0)|}{\|\omega\|_X}dx$ and $I_2:=\ln\|\omega\|_X\displaystyle\int_{\mathbb{R}^N}a(x)|\omega(x,0)|^2dx$.

If $\|\omega\|_X\leq1$, then (\ref{s2.1}) holds. If $\|\omega\|_X>1$, due to
$\displaystyle\int_{\mathbb{R}^N}a(x)|\omega(x,0)|^2dx\leq0$
we have $I_2\leq0.$
This implies
 \begin{equation}\label{s2.15}
 \displaystyle\int_{\mathbb{R}^N}a(x)|\omega(x,0)|^2\ln|\omega(x,0)|dx\leq I_1.
 \end{equation}

We divide $I_1$ into two parts:
$I_1=I_{11}+I_{12},$
where
$$I_{11}=\displaystyle\int_{\mathbb{R}_{a,+}^N}a(x)|\omega(x,0)|^2\ln
\frac{|\omega(x,0)|}{\|\omega\|_X}dx,\ \ I_{12}= +\displaystyle\int_{\mathbb{R}_{a,-}^N}a(x)|\omega(x,0)|^2\ln\frac{|\omega(x,0)|}{\|\omega\|_X}dx,$$
$$\mathbb{R}_{a,+}^N:=\left\{x\in\mathbb{R}^N;\, a(x)\geq0 \right\},\ \ \mathbb{R}_{a,-}^N:=\left\{x\in\mathbb{R}^N;\, a(x)<0\right\}.$$

When $t,\gamma>0$, by using $\ln t\leq C_\gamma t^\gamma$, it follows from Proposition \ref{p2.1} that
  \begin{equation}\label{s2.17}
 I_{11}\leq C\|\omega\|_X^{2-q}\displaystyle\int_{\mathbb{R}_{a,+}^N}a(x)|\omega(x,0)|^qdx\leq C\|\omega\|_X^2
 \end{equation}
 for  $2<q<2_\alpha^*$,
 and
 \begin{align}\label{s2.18}
I_{12}&\leq\displaystyle\int_{\Omega_{a,-}}a(x)|\omega(x,0)|^2\ln\frac{|\omega(x,0)|}{\|\omega\|_X}dx\notag
\\&=\displaystyle\int_{\Omega_{a,-}}(-a(x))|\omega(x,0)|^2\ln\frac{\|\omega\|_X}{|\omega(x,0)|}dx\notag
\\&\leq C_{\gamma_0}\|\omega\|_X^{\gamma_0}\displaystyle\int_{\mathbb{R}^N}|a(x)||\omega(x,0)|^{2-\gamma_0}dx\notag
\\&\leq C\|\omega\|_X^2
\end{align}
for  $0<\gamma_0<1$, where $\Omega_{a,-}:=\{x\in\mathbb{R}_{a,-}^N;|\omega(x,0)|<\|\omega\|_X\}$.

As a consequence of (\ref{s2.15})-(\ref{s2.18}), we obtain
\begin{equation*}
\displaystyle\int_{\mathbb{R}^N}a(x)|\omega(x,0)|^2\ln|\omega(x,0)|dx\leq C\|\omega\|_X^2,
 \end{equation*}
 where $C$ is a positive constant independent of $\omega\in  X^\alpha(\mathbb{R}_+^{N+1})$.
\end{proof}


\begin{lemma}\label{l2.1}
For each $\omega\in X^\alpha(\mathbb{R}_+^{N+1})\backslash\{0\}$, there exists  $\lambda_1>0$ small enough such that if $\lambda\in(0,\lambda_1)$,  then the following two statements are true.
\begin{description}
  \item[(i)] If $\displaystyle\int_{\mathbb{R}^N}a(x)|\omega(x,0)|^2dx>0$, then there exists   $t^-:=t^-(\omega)>0$ such that $t^-\omega\in N^-_\lambda$ and $I_\lambda(t^-\omega)=\displaystyle\max_{t\geq0}I_\lambda(t\omega)$.
  \item[(ii)] If $\displaystyle\int_{\mathbb{R}^N}a(x)|\omega(x,0)|^2dx<0$, then there exists a unique $0<t^+:=t^+(\omega)<t^-:=t^-(\omega)<\infty$ such that $t^+\omega\in N^+_\lambda$,\, $t^-\omega\in N^-_\lambda$,\, $I_\lambda(t\omega)$ is decreasing on $(0,t^+)$, increasing on $(t^+,t^-)$, and decreasing on $(t^-,+\infty)$. Moreover, $I_\lambda(t^+\omega)=\displaystyle\min_{0\leq t\leq t^-}I_\lambda(t\omega)$ and $I_\lambda(t^-\omega)=\displaystyle\max_{t^+\leq t}I_\lambda(t\omega)$.
\end{description}
\end{lemma}

\begin{proof}
$\textbf{(i)}$. Suppose that $\omega\in X^\alpha(\mathbb{R}_+^{N+1})\backslash\{0\}$ with $\displaystyle\int_{\mathbb{R}^N}a(x)|\omega(x,0)|^2dx>0$. Since $2<2_\alpha^*$ and
$\displaystyle\lim_{t\rightarrow 0^+}\ln t=-\infty,$
there exists $t_0>0$ small enough such that
 \begin{equation}\label{ee2.17}
\phi_\omega(t)>0
\end{equation}
for $t\in(0,t_0)$, where $\phi_\omega(t)$ is defined by (\ref{eee2.9}). Moreover, we have
\begin{equation}\label{ee2.18}
\displaystyle\lim_{t\rightarrow 0^+}\phi_\omega(t)=0~\text{and}~\displaystyle\lim_{t\rightarrow +\infty}\phi_\omega(t)=-\infty.
\end{equation}
From  (\ref{ee2.17}) with  (\ref{ee2.18}), there is  $t^-:=t^-(\omega)>0$ such that
\begin{equation*}
\phi_\omega(t^-)=I_\lambda(t^-\omega)=\displaystyle\max_{t\geq0}\phi_\omega(t)=\displaystyle\max_{t\geq0}I_\lambda(t\omega).
\end{equation*}
This implies $t^-\omega\in N^-_\lambda$.

$\textbf{(ii)}$. Suppose that $\omega\in X^\alpha(\mathbb{R}_+^{N+1})\backslash\{0\}$ with $\displaystyle\int_{\mathbb{R}^N}a(x)|\omega(x,0)|^2dx<0$.
Note that
\begin{align}
\frac{\phi'_\omega(t)}{t}&=\|\omega\|_X^2-\lambda\displaystyle\int_{\mathbb{R}^N}a(x)|\omega(x,0)|^2\ln|\omega(x,0)|dx\notag
 \\&~~~-\lambda\ln t\displaystyle\int_{\mathbb{R}^N}a(x)|\omega(x,0)|^2dx-t^{2_\alpha^*-2}\displaystyle\int_{\mathbb{R}^N}b(x)|\omega(x,0)|^{2_\alpha^*}dx.\notag
\end{align}

Let $s(t):=\lambda\ln t\displaystyle\int_{\mathbb{R}^N}a(x)|\omega(x,0)|^2dx+t^{2_\alpha^*-2}
\displaystyle\int_{\mathbb{R}^N}b(x)|\omega(x,0)|^{2_\alpha^*}dx.$
Then $t\omega\in N_\lambda$ if and only if
\begin{equation*}
s(t)=\|\omega\|_X^2-\lambda\displaystyle\int_{\mathbb{R}^N}a(x)|\omega(x,0)|^2\ln|\omega(x,0)|dx.
\end{equation*}

Since
\begin{equation}\label{ee2.19}
\displaystyle\lim_{t\rightarrow 0^+}s(t)=+\infty~\text{and}~\displaystyle\lim_{t\rightarrow +\infty}s(t)=+\infty
\end{equation}
and
\begin{equation}\label{ee2.20}
ts'(t)=\lambda\displaystyle\int_{\mathbb{R}^N}a(x)|\omega(x,0)|^2dx+(2_\alpha^*-2)t^{2_\alpha^*-2}\displaystyle\int_{\mathbb{R}^N}b(x)|\omega(x,0)|^{2_\alpha^*}dx,
\end{equation}
there exists
\begin{equation*}
t_{\min}:=\left(\frac{-\lambda\displaystyle\int_{\mathbb{R}^N}a(x)|\omega(x,0)|^2dx}{(2_\alpha^*-2)
\displaystyle\int_{\mathbb{R}^N}b(x)|\omega(x,0)|^{2_\alpha^*}dx}\right)^\frac{1}{2_\alpha^*-2}
\end{equation*}
such that
\begin{align}
&s(t_{\min})=\displaystyle\min_{t\geq0}s(t)\notag
\\&=\frac{1}{2_\alpha^*-2}\lambda\displaystyle\int_{\mathbb{R}^N}a(x)|\omega(x,0)|^2dx\ln\left(\frac{-\lambda\displaystyle\int_{\mathbb{R}^N}a(x)|\omega(x,0)|^2dx}{(2_\alpha^*-2)\displaystyle\int_{\mathbb{R}^N}b(x)|\omega(x,0)|^{2_\alpha^*}dx}\right)\notag
\\&~~~~-\frac{1}{2_\alpha^*-2}\lambda\displaystyle\int_{\mathbb{R}^N}a(x)|\omega(x,0)|^2dx.\notag
\end{align}
Moreover, $s(t)$ is decreasing in $(0,t_{\min})$ and increasing in $(t_{\min},+\infty)$.

To show that
\begin{equation}\label{ee2.21}
s(t_{\min})<\|\omega\|_X^2-\lambda\displaystyle\int_{\mathbb{R}^N}a(x)|\omega(x,0)|^2\ln|\omega(x,0)|dx,
\end{equation}
we start with estimating $\displaystyle\int_{\mathbb{R}^N}a(x)|\omega(x,0)|^2\ln|\omega(x,0)|dx$.  It follows from Lemma \ref{l2.0} that either
\begin{equation}\label{ee2.22}
\|\omega\|_X\leq 1
 \end{equation}
 or
 \begin{equation}\label{ee2.23}
\displaystyle\int_{\mathbb{R}^N}a(x)|\omega(x,0)|^2\ln|\omega(x,0)|dx\leq C\|\omega\|_X^2,
 \end{equation}
 for some $C>0$ independent of $\omega$. Thus we need to consider two cases.

 \textbf{Case 1}. Assume that (\ref{ee2.22}) holds. On the one hand, it follows from (\ref{e2.6}) and Proposition \ref{p2.1}  that
  \begin{equation*}
\displaystyle\int_{\mathbb{R}^N}a(x)|\omega(x,0)|^2\ln|\omega(x,0)|dx\leq C(\|\omega\|_X^2+\|\omega\|_X^4)\leq C\|\omega\|_X^2
 \end{equation*}
for  some $C>0$ independent of $\omega$. So we have
 \begin{equation}\label{ee2.25}
\|\omega\|_X^2-\lambda\displaystyle\int_{\mathbb{R}^N}a(x)|\omega(x,0)|^2\ln|\omega(x,0)|dx\geq C\|\omega\|_X^2
 \end{equation}
for $\lambda>0$ small enough and some $C>0$ independent of $\omega$.

On the other hand, in view of the inequality $\ln t\leq t$ for $t>0$, it follows from Proposition \ref{p2.1} and (\ref{ee2.22}) that
\begin{align}\label{ee2.26}
&s(t_{\min})\notag
\\=&\frac{1}{2_\alpha^*-2}\lambda\displaystyle\int_{\mathbb{R}^N}a(x)|\omega(x,0)|^2dx\ln\left(-\lambda\displaystyle\int_{\mathbb{R}^N}a(x)|\omega(x,0)|^2dx\right)\notag
\\&-\frac{1}{2_\alpha^*-2}\lambda\displaystyle\int_{\mathbb{R}^N}a(x)|\omega(x,0)|^2dx\ln\left((2_\alpha^*-2)\displaystyle\int_{\mathbb{R}^N}b(x)|\omega(x,0)|^{2_\alpha^*}dx\right)\notag
\\&-\frac{1}{2_\alpha^*-2}\lambda\displaystyle\int_{\mathbb{R}^N}a(x)|\omega(x,0)|^2dx\notag
\\ \leq& \frac{-\lambda}{2_\alpha^*-2}\displaystyle\int_{\mathbb{R}^N}a(x)|\omega(x,0)|^2dx\left[-\lambda\displaystyle\int_{\mathbb{R}^N}a(x)|\omega(x,0)|^2dx+(2_\alpha^*-2)\displaystyle\int_{\mathbb{R}^N}b(x)|\omega(x,0)|^{2_\alpha^*}dx+1\right]\notag\\
\leq& \lambda C\|\omega\|_X^2[\lambda C\|\omega\|_X^2+C\|\omega\|_X^{2_\alpha^*}+1]\notag\\
\leq& \lambda C\|\omega\|_X^2
\end{align}
for  some $C>0$ independent of $\omega$.  As a consequence of (\ref{ee2.25}) and (\ref{ee2.26}), we see that (\ref{ee2.21}) holds for $\lambda>0$ small enough.

From (\ref{ee2.19})-(\ref{ee2.21}),  there exists a unique $0<t^+(\omega)<t_{\min}<t^-(\omega)<\infty$ such that
\begin{equation*}
s(t^+(\omega))=s(t^-(\omega))=\|\omega\|_X^2-\lambda\displaystyle\int_{\mathbb{R}^N}a(x)|\omega(x,0)|^2\ln|\omega(x,0)|dx
\end{equation*}
 and
 \begin{equation*}
t^+(\omega)\omega\in N_\lambda~\text{and}~t^-(\omega)\omega\in N_\lambda.
 \end{equation*}
Since
\begin{equation*}
s'(t^+(\omega))<0<s'(t^-(\omega)),
 \end{equation*}
it follows from (\ref{ee2.20}) that $t^+(\omega)\omega\in N^+_\lambda$ and $t^-(\omega)\omega\in N^-_\lambda.$

Using the fact that
\begin{displaymath}
s(t)-\|\omega\|_X^2-\lambda\displaystyle\int_{\mathbb{R}^N}a(x)|\omega(x,0)|^2\ln|\omega(x,0)|dx\left\{\begin{array}{ll}
\geq0,~ 0\leq t\leq t^+(\omega),\\
\leq0,~ t^+(\omega)\leq t\leq t^-(\omega),\\
\geq0,~t^-(\omega)\leq t,
\end{array}\right.
\end{displaymath}
we obtain
\begin{displaymath}
\phi'_\omega(t)\left\{\begin{array}{ll}
\leq0,~\, 0\leq t\leq t^+(\omega),\\
\geq0,~\, t^+(\omega)\leq t\leq t^-(\omega),\\
\leq0,\, ~t^-(\omega)\leq t.
\end{array}\right.
\end{displaymath}
This indicates that $I_\lambda(t\omega)$ is decreasing on $(0,t^+(\omega))$, increasing on $(t^+(\omega),t^-(\omega))$ and decreasing on $(t^-(\omega),\infty)$. Moreover,
we have
\begin{equation*}
I_\lambda(t^+(\omega)\omega)=\displaystyle\min_{0\leq t\leq t^-(\omega)}I_\lambda(t\omega)~\text{and}~ I_\lambda(t^-\omega)=\displaystyle\max_{t\geq t^+(\omega)}I_\lambda(t\omega).
\end{equation*}

 \textbf{Case 2}. Assume that (\ref{ee2.23}) holds. Then
  \begin{equation}\label{ee2.27}
\|\omega\|_X^2-\lambda\displaystyle\int_{\mathbb{R}^N}a(x)|\omega(x,0)|^2\ln|\omega(x,0)|dx\geq C\|\omega\|_X^2
 \end{equation}
for $\lambda>0$ small enough and   some $C>0$ independent of $\omega$.

If $-\lambda\displaystyle\int_{\mathbb{R}^N}a(x)|\omega(x,0)|^2dx\geq(2_\alpha^*-2)\displaystyle\int_{\mathbb{R}^N}b(x)|\omega(x,0)|^{2_\alpha^*}dx$, we have
 \begin{equation*}
\frac{\lambda}{2_\alpha^*-2}\displaystyle\int_{\mathbb{R}^N}a(x)|\omega(x,0)|^2dx\ln\left(\frac{-\lambda\displaystyle\int_{\mathbb{R}^N}a(x)|\omega(x,0)|^2dx}{(2_\alpha^*-2)
\displaystyle\int_{\mathbb{R}^N}b(x)|\omega(x,0)|^{2_\alpha^*}dx}\right)\leq0.
 \end{equation*}
That is,
 \begin{equation}\label{ee2.28}
s(t_{\min})\leq\frac{-\lambda}{2_\alpha^*-2}\displaystyle\int_{\mathbb{R}^N}a(x)|\omega(x,0)|^2dx\leq\lambda C\|\omega\|_X^2
 \end{equation}
for some $C>0$ independent of $\omega$. From (\ref{ee2.27}) and (\ref{ee2.28}), we arrive at the desired result (\ref{ee2.21}) for $\lambda>0$ small enough.

Due to (\ref{ee2.21}),  similar to the proof of Case 1, there exist $0<t^+(\omega)<t_{\min}<t^-(\omega)<\infty$ such that  $t^+(\omega)\omega\in N^+_\lambda$ and $t^-(\omega)\omega\in N^-_\lambda.$  Furthermore, we can see that $I_\lambda(t\omega)$ is decreasing on $(0,t^+(\omega))$, increasing on $(t^+(\omega),t^-(\omega))$ and decreasing on $(t^-(\omega),\infty)$. So we have $I_\lambda(t^+(\omega)\omega)=\displaystyle\min_{0\leq t\leq t^-(\omega)}I_\lambda(t\omega)$ and $I_\lambda(t^-\omega)=\displaystyle\max_{t\geq t^+(\omega)}I_\lambda(t\omega)$.

If $-\lambda\displaystyle\int_{\mathbb{R}^N}a(x)|\omega(x,0)|^2dx<(2_\alpha^*-2)\displaystyle\int_{\mathbb{R}^N}b(x)|\omega(x,0)|^{2_\alpha^*}dx$, since $2<2_\alpha^*$,  there exists $t_0>0$ such that
\begin{equation*}
-\lambda\displaystyle\int_{\mathbb{R}^N}a(x)|t_0\omega(x,0)|^2dx>(2_\alpha^*-2)\displaystyle\int_{\mathbb{R}^N}b(x)|t_0\omega(x,0)|^{2_\alpha^*}dx.
 \end{equation*}

Set
 \begin{equation*}
\omega_0=t_0\omega.
 \end{equation*}
Similarly,  we can see that there are $0<t^+(\omega_0)<t^-(\omega_0)<\infty$ such that the desired result in Case 1 holds for some $\omega_0$.
Let $t^+(\omega)=t_0t^+(\omega_0)\ \text{and} \ t^-(\omega)=t_0t^-(\omega_0).$
Consequently, there exist $0<t^+(\omega)<t^-(\omega)<\infty$ such that the result in Case 1 holds for an arbitrary $\omega$.
\end{proof}

\begin{lemma}\label{l2.2}
If $\omega$ is a critical point of $I_\lambda$ on $N_\lambda$ and $\omega\not\in N^0_\lambda$, then it is a critical point of $I_\lambda$ in $X^\alpha(\mathbb{R}_+^{N+1})$.
\end{lemma}

\begin{proof}
Let $\omega$ be a critical point of $I_\lambda$ on $N_\lambda$. Then
\begin{equation}\label{e2.9}
I'_\lambda(\omega)\omega=0\ \text{and}\ I'_\lambda(\omega)=\tau\Psi'_\lambda(\omega)
\end{equation}
for some $\tau\in \mathbb{R}$, where
 \begin{equation}\label{ee2.9}
\Psi_\lambda(\omega):=\|\omega\|_X^2-\lambda\displaystyle\int_{\mathbb{R}^N}a(x)|\omega(x,0)|^2\ln|\omega(x,0)|dx-\displaystyle\int_{\mathbb{R}^N}b(x)|\omega(x,0)|^{2_\alpha^*}dx.
 \end{equation}

Because of $\omega\not\in N^\pm_\lambda$, we get
 \begin{align}
\Psi_\lambda'(\omega)\omega:&=2\|\omega\|^2_X-2\lambda\displaystyle\int_{\mathbb{R}^N}a(x)|\omega(x,0)|^2\ln|\omega(x,0)|dx\notag
\\&~~~~-\lambda\displaystyle\int_{\mathbb{R}^N}a(x)|\omega(x,0)|^2dx-2_\alpha^*\displaystyle\int_{\mathbb{R}^N}b(x)|\omega(x,0)|^{2_\alpha^*}dx\notag
\\&=\|\omega\|^2_X-\lambda\displaystyle\int_{\mathbb{R}^N}a(x)|\omega(x,0)|^2\ln|\omega(x,0)|dx\notag
\\&~~~~-\lambda\displaystyle\int_{\mathbb{R}^N}a(x)|\omega(x,0)|^2dx-(2_\alpha^*-1)\displaystyle\int_{\mathbb{R}^N}b(x)|\omega(x,0)|^{2_\alpha^*}dx\notag
\\&=\phi''_\omega(1)\neq0,\notag
 \end{align}
 which together with (\ref{e2.9}) indicates that $\tau=0$, i.e. $I'_\lambda(\omega)=0$.
\end{proof}


\begin{lemma}\label{l2.3}
There exists $\lambda_2>0$ small enough such that if $\lambda\in(0,\lambda_2)$, then the set $N^0_\lambda=\emptyset$.
\end{lemma}
\begin{proof}
Let $\omega\in N^0_\lambda$. Then
\begin{equation}\label{e2.10}
\|\omega\|_X^2-\lambda\displaystyle\int_{\mathbb{R}^N}a(x)|\omega(x,0)|^2\ln|\omega(x,0)|dx-\displaystyle\int_{\mathbb{R}^N}b(x)|\omega(x,0)|^{2_\alpha^*}dx=0
\end{equation}
and
\begin{align}
 &\|\omega\|_X^2-\lambda\displaystyle\int_{\mathbb{R}^N}a(x)|\omega(x,0)|^2\ln|\omega(x,0)|dx
-\lambda\displaystyle\int_{\mathbb{R}^N}a(x)|\omega(x,0)|^2dx \notag\\ &-(2_\alpha^*-1)\displaystyle\int_{\mathbb{R}^N}b(x)|\omega(x,0)|^{2_\alpha^*}dx=0,\notag
\end{align}
which lead to
\begin{equation}\label{e2.11}
\lambda\displaystyle\int_{\mathbb{R}^N}a(x)|\omega(x,0)|^2dx=(2-2_\alpha^*)\displaystyle\int_{\mathbb{R}^N}b(x)|\omega(x,0)|^{2_\alpha^*}dx<0.
\end{equation}

In view of (\ref{e2.11}), it follows from Lemma \ref{l2.0} that either
 \begin{equation}\label{e2.12}
 \|\omega\|_X\leq1,
\end{equation}
or
\begin{equation}\label{e2.19}
\displaystyle\int_{\mathbb{R}^N}a(x)|\omega(x,0)|^2\ln|\omega(x,0)|dx\leq C\|\omega\|_X^2,
 \end{equation}
 where $C$ is a positive constant independent of $\omega\in N_\lambda^0$. If (\ref{e2.19}) holds, for sufficiently small $\lambda>0$ it follows from  (\ref{e2.10})-(\ref{e2.11}) and Proposition \ref{p2.1} that
 \begin{align}
0&=\|\omega\|_X^2-\lambda\displaystyle\int_{\mathbb{R}^N}a(x)|\omega(x,0)|^2\ln|\omega(x,0)|dx-\displaystyle\int_{\mathbb{R}^N}b(x)|\omega(x,0)|^{2_\alpha^*}dx\notag
\\&=\|\omega\|_X^2-\lambda\displaystyle\int_{\mathbb{R}^N}a(x)|\omega(x,0)|^2\ln|\omega(x,0)|dx+\frac{\lambda}{2_\alpha^*-2}\displaystyle\int_{\mathbb{R}^N}a(x)|\omega(x,0)|^2dx\notag
\\&\geq \|\omega\|_X^2(1-\lambda C) \notag \\
&\geq C\|\omega\|_X^2. \notag
\end{align}
Thus, $\|\omega\|_X=0$, which obviously yields a contradiction to the fact $\omega\neq0$. This implies that (\ref{e2.12}) holds.

On the other hand, in view of $ \ln t\leq  t$ for any $t>0$, it follows from (\ref{e2.6}) and Proposition \ref{p2.1} that
  \begin{align}\label{eee2.20}
\displaystyle\int_{\mathbb{R}^N}a(x)|\omega(x,0)|^2\ln|\omega(x,0)|dx&\leq C(\|\omega\|_X^2+|\omega(x,0)|_2^4)\leq C(\|\omega\|_X^2+\|\omega\|_X^4).
\end{align}
With the help of (\ref{e2.10})-(\ref{e2.11}), (\ref{eee2.20}) and Proposition \ref{p2.1},  we obtain
\begin{align}
\|\omega\|_X^2&=\lambda\displaystyle\int_{\mathbb{R}^N}a(x)|\omega(x,0)|^2\ln|\omega(x,0)|dx-\frac{\lambda}{2_\alpha^*-2}\displaystyle\int_{\mathbb{R}^N}a(x)|\omega(x,0)|^2dx\notag
\\&\leq\lambda C(\|\omega\|_X^2+\|\omega\|_X^4),\notag
\end{align}
which together with (\ref{e2.12}) gives
\begin{equation*}
 C\leq \lambda (1+\|\omega\|_X^2)\leq 2\lambda.
 \end{equation*}
 This is a contradiction with the fact that $\lambda$ is sufficiently small.
\end{proof}

\begin{lemma}\label{l2.4}
There exists $\lambda_3>0$ small enough such that if $\lambda\in(0,\lambda_3)$, then $I_\lambda$ is bounded from below on $N_\lambda$.
\end{lemma}

\begin{proof}
Let $\omega\in N^+_\lambda$. According to the definition of $N^+_\lambda$, we get
\begin{equation*}
\lambda\displaystyle\int_{\mathbb{R}^N}a(x)|\omega(x,0)|^2dx<0
 \end{equation*}
 and
 \begin{equation*}
\displaystyle\int_{\mathbb{R}^N}b(x)|\omega(x,0)|^{2_\alpha^*}dx<\frac{-\lambda}{2_\alpha^*-2}\displaystyle\int_{\mathbb{R}^N}a(x)|\omega(x,0)|^2dx.
 \end{equation*}
 As discussing for (\ref{e2.12}), for $\lambda>0$ small enough we can obtain
 \begin{equation}\label{e2.21}
\|\omega\|_X\leq1.
 \end{equation}
Hence, the low bound of $I_\lambda$ restricted on $N^+_\lambda$ can be attained by Proposition \ref{p2.1} and (\ref{e2.21}), i.e.
 \begin{align}\label{e2.23}
I_\lambda(\omega)&=I_\lambda(\omega)-\frac{1}{2}I'_\lambda(\omega)\omega\notag
\\&=\frac{\lambda}{4}\displaystyle\int_{\mathbb{R}^N}a(x)|\omega(x,0)|^2dx+\left(\frac{1}{2}-\frac{1}{2_\alpha^*}\right)\displaystyle\int_{\mathbb{R}^N}b(x)|\omega(x,0)|^{2_\alpha^*}dx\notag
\\&\geq\frac{\lambda}{4}\displaystyle\int_{\mathbb{R}^N}a(x)|\omega(x,0)|^2dx\notag
\\&\geq-\lambda C\|\omega\|_X^2\notag
\\&\geq-\lambda C.
\end{align}

For any $\omega\in N^-_\lambda$, we have
\begin{align}\label{e2.24}
I_\lambda(\omega)&=I_\lambda(\omega)-\frac{1}{2}I'_\lambda(\omega)\omega\nonumber\\&=\frac{\lambda}{4}\displaystyle\int_{\mathbb{R}^N}a(x)|\omega(x,0)|^2dx+\left(\frac{1}{2}-\frac{1}{2_\alpha^*}\right)\displaystyle\int_{\mathbb{R}^N}b(x)|\omega(x,0)|^{2_\alpha^*}dx.
 \end{align}
If $I_\lambda(\omega)\geq0$ for all $\omega\in N^-_\lambda$, obviously the lower bound of $I_\lambda$ restricted on $N^-_\lambda$ can be achieved. Otherwise, if there exists $\omega\in N^-_\lambda$ such that $I_\lambda(\omega)<0$, by (\ref{e2.24}) it follows that
 \begin{equation*}
\lambda\displaystyle\int_{\mathbb{R}^N}a(x)|\omega(x,0)|^2dx<0
 \end{equation*}
 and
 \begin{equation*}
\displaystyle\int_{\mathbb{R}^N}b(x)|\omega(x,0)|^{2_\alpha^*}dx<\frac{-2_\alpha^*}{2(2_\alpha^*-2)}\lambda\displaystyle\int_{\mathbb{R}^N}a(x)|\omega(x,0)|^2dx.
 \end{equation*}

As we did for (\ref{e2.12}), there is $\lambda>0$ small enough such that
 \begin{equation}\label{e2.25}
\|\omega\|_X\leq1.
 \end{equation}
Using Proposition \ref{p2.1} and (\ref{e2.25}), we obtain
 \begin{align}\label{e2.26}
I_\lambda(\omega)&=\frac{\lambda}{4}\displaystyle\int_{\mathbb{R}^N}a(x)|\omega(x,0)|^2dx+\left(\frac{1}{2}-\frac{1}{2_\alpha^*}\right)\displaystyle\int_{\mathbb{R}^N}b(x)|\omega(x,0)|^{2_\alpha^*}dx\notag
\\&\geq\frac{\lambda}{4}\displaystyle\int_{\mathbb{R}^N}a(x)|\omega(x,0)|^2dx\notag
\\&\geq-\lambda C\|\omega\|_X^2\notag
\\&\geq-\lambda C.\notag
\end{align}
Hence, $I_\lambda$ is bounded from below on $N^-_\lambda$ according to Lemma \ref{l2.3}.
\end{proof}

In view of Lemmas \ref{l2.1} and \ref{l2.4}, we set
\begin{equation*}
\alpha_\lambda^+:=\displaystyle\inf_{\omega\in N_\lambda^+}I_\lambda(\omega)\ \text{and}\ \alpha_\lambda^-:=\displaystyle\inf_{\omega\in N_\lambda^-}I_\lambda(\omega).
 \end{equation*}

\begin{lemma}\label{l2.5}
(i) If $a(x)$ is negative or sign-changing, then $\alpha_\lambda^+<0$ and $\alpha_\lambda^+\leq\alpha_\lambda^-$.\\
(ii) If $a(x)\geq0$, then $N_\lambda^+=\emptyset$ and $\alpha_\lambda^->0$.
\end{lemma}
\begin{proof}
$\textbf{(i)}$. If $a(x)$ is negative or sign-changing, it follows from Lemma \ref{l2.1} that $N_\lambda^+\neq\emptyset$.
Let $\omega\in N^+_\lambda$. Then we have
\begin{equation*}
\lambda\displaystyle\int_{\mathbb{R}^N}a(x)|\omega(x,0)|^2dx<(2-2_\alpha^*)\displaystyle\int_{\mathbb{R}^N}b(x)|\omega(x,0)|^{2_\alpha^*}dx<0.
 \end{equation*}
This, together with $I'_\lambda(\omega)\omega=0$,  leads to
 \begin{align}
I_\lambda(\omega)&=I_\lambda(\omega)-\frac{1}{2}I'_\lambda(\omega)\omega\notag
\\&=\frac{\lambda}{4}\displaystyle\int_{\mathbb{R}^N}a(x)|\omega(x,0)|^2dx+\left(\frac{1}{2}-\frac{1}{2_\alpha^*}\right)\displaystyle\int_{\mathbb{R}^N}b(x)|\omega(x,0)|^{2_\alpha^*}dx\notag
\\&<\frac{2-2_\alpha^*}{4}\displaystyle\int_{\mathbb{R}^N}b(x)|\omega(x,0)|^{2_\alpha^*}dx+\frac{2_\alpha^*-2}{22_\alpha^*}\displaystyle\int_{\mathbb{R}^N}b(x)|\omega(x,0)|^{2_\alpha^*}dx\notag
\\&=\left(\frac{1}{4}-\frac{1}{22_\alpha^*}\right)(2-2_\alpha^*)\displaystyle\int_{\mathbb{R}^N}b(x)|\omega(x,0)|^{2_\alpha^*}dx
\\&<0.\notag
\end{align}
Thus, we obtain $\alpha_\lambda^+<0$.

For any $\omega\in N^-_\lambda$, if $I_\lambda(\omega)\geq0$, then
\begin{equation}\label{s2.39}
I_\lambda(\omega)\geq\alpha_\lambda^+.
 \end{equation}
 If $I_\lambda(\omega)<0$, then
\begin{align*}
I_\lambda(\omega)&=I_\lambda(\omega)-\frac{1}{2}I'_\lambda(\omega)\omega
\nonumber\\&=\frac{\lambda}{4}\displaystyle\int_{\mathbb{R}^N}a(x)|\omega(x,0)|^2dx+\left(\frac{1}{2}-\frac{1}{2_\alpha^*}\right)\displaystyle\int_{\mathbb{R}^N}b(x)|\omega(x,0)|^{2_\alpha^*}dx
\nonumber\\&<0.
 \end{align*}
That is,
\begin{equation*}
\displaystyle\int_{\mathbb{R}^N}a(x)|\omega(x,0)|^2dx<0.
 \end{equation*}
With the help of Lemma \ref{l2.1} $(ii)$, there exists a unique $t^+(\omega)<t^-(\omega)=1$ such that $t^+(\omega)\omega\in N_\lambda^+$ and
\begin{equation}\label{s2.40}
I_\lambda(\omega)\geq I_\lambda(t^+(\omega)\omega)\geq \alpha_\lambda^+.
 \end{equation}
Consequently,  as a result of (\ref{s2.39}) and (\ref{s2.40}), we obtain
\begin{equation*}
\alpha_\lambda^+\leq\alpha_\lambda^-.
 \end{equation*}

$\textbf{(ii)}$. If $a(x)\geq0$, then for any $\omega\in X^\alpha(\mathbb{R}_+^{N+1})\backslash\{0\}$ we have
\begin{equation}\label{ss2.35}
\displaystyle\int_{\mathbb{R}^N}a(x)|\omega(x,0)|^2dx\geq0,
 \end{equation}
which implies $N_\lambda^+=\emptyset$.
Moreover, it follows from Lemma \ref{l2.1} $(i)$ that $N_\lambda^-\neq\emptyset$.

For any $\omega\in N^-_\lambda$, we get
\begin{align*}
I_\lambda(\omega)&=I_\lambda(\omega)-\frac{1}{2}I'_\lambda(\omega)\omega
\nonumber\\&=\frac{\lambda}{4}\displaystyle\int_{\mathbb{R}^N}a(x)|\omega(x,0)|^2dx+\left(\frac{1}{2}-\frac{1}{2_\alpha^*}\right)\displaystyle\int_{\mathbb{R}^N}b(x)|\omega(x,0)|^{2_\alpha^*}dx
\nonumber\\&\geq0,
 \end{align*}
 which implies
 \begin{equation*}
\alpha_\lambda^-\geq0.
 \end{equation*}

 We now suppose by contradiction that $\alpha_\lambda^-=0$. Let $\{\omega_n\}\subset N_\lambda^-$ be a sequence such that $I_\lambda(\omega_n)\rightarrow0$, as $n\rightarrow\infty$. Then we have
 \begin{align}
0\leftarrow I_\lambda(\omega_n)&=I_\lambda(\omega_n)-\frac{1}{2}I'_\lambda(\omega_n)\omega_n\notag
\\&=\frac{\lambda}{4}\displaystyle\int_{\mathbb{R}^N}a(x)|\omega_n(x,0)|^2dx+\left(\frac{1}{2}-\frac{1}{2_\alpha^*}\right)\displaystyle\int_{\mathbb{R}^N}b(x)|\omega_n(x,0)|^{2_\alpha^*}dx
\nonumber\\&\geq 0,\ \text{as}\ n \rightarrow \infty, \notag
 \end{align}
 which together with (\ref{ss2.35}) yields
 \begin{equation}\label{ee2.13}
\lambda\displaystyle\int_{\mathbb{R}^N}a(x)|\omega_n(x,0)|^2dx=o_n(1)~\text{and}~\displaystyle\int_{\mathbb{R}^N}b(x)|\omega_n(x,0)|^{2_\alpha^*}dx=o_n(1).
 \end{equation}

It follows (\ref{ee2.13}) and Proposition \ref{p2.1} that
\begin{align}\label{ee2.14}
&\displaystyle\int_{\mathbb{R}^N}a(x)|\omega_n(x,0)|^2\ln|\omega_n(x,0)|dx\notag
\\=& \displaystyle\int_{\mathbb{R}^N}a(x)|\omega_n(x,0)|^2\ln\frac{|\omega_n(x,0)|}{\|\omega_n\|_X}dx+\ln\|\omega_n\|_X\displaystyle\int_{\mathbb{R}^N}a(x)|\omega_n(x,0)|^2dx\notag
\\ \leq & \displaystyle\int_{\mathbb{R}^N}a(x)|\omega_n(x,0)|^2\ln\frac{|\omega_n(x,0)|}{\|\omega_n\|_X}dx+C\|\omega_n\|_X^2.
\end{align}
Processing as we did for (\ref{s2.17}) and (\ref{s2.18}), we have
 \begin{equation*}
\displaystyle\int_{\mathbb{R}^N}a(x)|\omega_n(x,0)|^2\ln\frac{|\omega_n(x,0)|}{\|\omega_n\|_X}dx\leq C\|\omega_n\|_X^2.
 \end{equation*}
Using this estimate together with (\ref{ee2.14}) leads to
  \begin{equation}\label{eee2.16}
\displaystyle\int_{\mathbb{R}^N}a(x)|\omega_n(x,0)|^2\ln|\omega_n(x,0)|dx\leq C\|\omega_n\|_X^2.
 \end{equation}
Taking into account (\ref{e2.8}), (\ref{ee2.13}), (\ref{eee2.16})  and Proposition \ref{p2.1}, for sufficiently small $\lambda>0$ we deduce that
 \begin{align*}
0&=\|\omega_n\|_X^2-\lambda\displaystyle\int_{\mathbb{R}^N}a(x)|\omega_n(x,0)|^2\ln|\omega_n(x,0)|dx-\displaystyle\int_{\mathbb{R}^N}b(x)|\omega_n(x,0)|^{2_\alpha^*}dx\notag
\\&=\|\omega_n\|_X^2-\lambda\displaystyle\int_{\mathbb{R}^N}a(x)|\omega_n(x,0)|^2\ln|\omega_n(x,0)|dx+o_n(1)\notag
\\&\geq \|\omega_n\|_X^2(1-\lambda C)+o_n(1)\\&\geq C\|\omega_n\|_X^2++o_n(1). \notag
\end{align*}
That is,
\begin{equation}\label{ee2.16}
\|\omega_n\|_X=o_n(1).
 \end{equation}

On the other hand, in view of $\ln t\leq  t$ for $t>0$, it follows from (\ref{e2.6}) and Proposition \ref{p2.1} that
  \begin{align}\label{e2.20}
\displaystyle\int_{\mathbb{R}^N}a(x)|\omega_n(x,0)|^2\ln|\omega_n(x,0)|dx&\leq C\left(\|\omega_n\|_X^2+|\omega_n(x,0)|_2^4 \right)\leq C \left(\|\omega_n\|_X^2+\|\omega_n\|_X^4\right).
\end{align}
Making use of (\ref{e2.8}), (\ref{e2.20}) and Proposition \ref{p2.2}, we get
\begin{align}
\|\omega_n\|_X^2&=\lambda\displaystyle\int_{\mathbb{R}^N}a(x)|\omega_n(x,0)|^2\ln|\omega_n(x,0)|dx+\displaystyle\int_{\mathbb{R}^N}b(x)|\omega_n(x,0)|^{2_\alpha^*}dx\notag
\\&\leq\lambda C(\|\omega_n\|_X^2+\|\omega_n\|_X^4)+C\|\omega_n\|_X^{2_\alpha^*}.\notag
\end{align}
That is, $
 \|\omega_n\|_X^{2_\alpha^*}+\|\omega_n\|_X^4\geq (1-\lambda C)\|\omega_n\|_X^2\geq C\|\omega_n\|_X^2$
for small $\lambda>0$ and some $C>0$. Hence, we have $
 \|\omega_n\|_X^2\geq C$ for some $C>0$ independent of $n\in\mathbb{Z}_+$. Apparently, this yields
 a contradiction to  (\ref{ee2.16}).
\end{proof}

Now we are ready to prove Theorem \ref{t1.1}.

\begin{proof}[Proof of Theorem \ref{t1.1}.]
\textbf{Step 1.} We shall show that there exists $\Lambda_1>0$ such that for each $\lambda\in(0,\Lambda_1)$, $I_\lambda$ has a minimizer $\omega_\lambda^+$ in $N_\lambda^+$ such that $I_\lambda(\omega_\lambda^+)=\alpha_\lambda^+$.

Let $\{\omega_n\}$ be a minimizing sequence $\{\omega_n\}\subset N_\lambda^+$, i.e.
$\displaystyle\lim_{n\rightarrow\infty}I_\lambda(\omega_n)=\alpha_\lambda^+.$
 We claim that there is some $C>0$ such that
 \begin{equation}\label{e2.26}
\|\omega_n\|_X\leq C
 \end{equation}
 for all $n\in\mathbb{Z}_+$. Note that $\{\omega_n\}\subset N_\lambda^+$. Then
 \begin{equation*}
\displaystyle\int_{\mathbb{R}^N}a(x)|\omega_n(x,0)|^2dx<0
 \end{equation*}
 and
  \begin{equation*}
\displaystyle\int_{\mathbb{R}^N}b(x)|\omega_n(x,0)|^{2_\alpha^*}dx<\frac{-\lambda}{2_\alpha^*-2}\displaystyle\int_{\mathbb{R}^N}a(x)|\omega_n(x,0)|^2dx.
 \end{equation*}
Analogous to the derivation of (\ref{e2.12}), we can see that (\ref{e2.26}) holds for $\lambda>0$ small enough.
 Thus there exists a subsequence (still denoted by $\{\omega_n\}$) and $\omega_\lambda^+\in X^\alpha(\mathbb{R}^{N+1}_+)$ such that
 \begin{equation}\label{ss2.42}
\omega_n\rightharpoonup\omega_\lambda^+\ \text{in}\ X^\alpha(\mathbb{R}^{N+1}_+), \ \text{as} \ n\rightarrow \infty.
 \end{equation}

To prove that
 \begin{equation}\label{e2.27}
\displaystyle\int_{\mathbb{R}^N}a(x)|\omega_n(x,0)|^2dx\rightarrow
\displaystyle\int_{\mathbb{R}^N}a(x)|\omega_\lambda^+(x,0)|^2dx, \ \text{as} \ n\rightarrow \infty,
 \end{equation}
we know that for any $\varepsilon>0$, according to condition $(H_1)$, there exists $R>0$ such that
$|a(x)|<\varepsilon$ for  $|x|\geq R$. It follows from (\ref{e2.26}) and Proposition \ref{p2.1} that
 \begin{equation}\label{e2.28}
\left|\displaystyle\int_{\mathbb{R}^N\backslash B_R}a(x)|\omega_n(x,0)|^2dx\right|\leq\varepsilon\|\omega_n\|_X^2\leq C\varepsilon
 \end{equation}
 and
 \begin{equation}\label{e2.29}
\left|\displaystyle\int_{\mathbb{R}^N\backslash B_R}a(x)|\omega_\lambda^+(x,0)|^2dx\right|\leq\varepsilon\|\omega_\lambda^+\|_X^2\leq C\varepsilon,
 \end{equation}
 where $B_R:=\{x\in\mathbb{R}^N;|x|<R\}$.
Then, using H\"{o}lder's inequality and Proposition \ref{p2.1} leads to
 \begin{equation}\label{e2.30}
\left|\displaystyle\int_{B_R}a(x)|\omega_n(x,0)|^2dx-
\displaystyle\int_{B_R}a(x)|\omega_\lambda^+(x,0)|^2dx\right|\rightarrow 0, \ \text{as} \ n\rightarrow \infty.
 \end{equation}
From (\ref{e2.28})-(\ref{e2.30}), we arrive at (\ref{e2.27}).

To prove that
 \begin{equation}\label{e2.31}
\displaystyle\int_{\mathbb{R}^N}a(x)|\omega_n(x,0)|^2\ln|\omega_n(x,0)|dx\rightarrow
\displaystyle\int_{\mathbb{R}^N}a(x)|\omega_\lambda^+(x,0)|^2\ln|\omega_\lambda^+(x,0)|dx, \ \text{as} \ n\rightarrow \infty,
 \end{equation}
we know from (\ref{ss2.42}) that $\omega_n(x,0)\rightarrow\omega_\lambda^+(x,0)$ for a.e. $x\in\mathbb{R}^N$. So for sufficiently large $n$ there holds
  \begin{equation*}
a(x)|\omega_n(x,0)|^2\ln|\omega_n(x,0)|\rightarrow a(x)|\omega_\lambda^+(x,0)|^2\ln|\omega_\lambda^+(x,0)|
 \end{equation*}
 for a.e. $x\in\mathbb{R}^N$. Note that for any $\beta,\,\gamma>0$, there exists a constant $C_{\beta,\gamma}>0$ such that
 \begin{equation*}
|\ln t|\leq C_{\beta,\gamma} \left(t^\beta+t^{-\gamma} \right),\ t>0.
 \end{equation*}
This gives
 \begin{equation*}
\left|\displaystyle\int_{\mathbb{R}^N}a(x)|\omega_n(x,0)|^2\ln|\omega_n(x,0)|dx\right|\leq C\displaystyle\int_{\mathbb{R}^N}|a(x)|\big[|\omega_n(x,0)|^{2-\delta}+|\omega_n(x,0)|^{2+\delta}\big]dx
 \end{equation*}
 for small $\delta>0$. By virtue of Proposition \ref{p2.1} and Lebesgue's dominated convergence theorem, we obtain (\ref{e2.31}) immediately.

 Set $\Psi_n=\omega_n-\omega_\lambda^+$. It follows from Brezis-Lieb's lemma  \cite{32} that
 \begin{equation}\label{e2.32}
\|\Psi_n\|_X^2=\|\omega_n\|_X^2-\|\omega_\lambda^+\|_X^2+o_n(1)
 \end{equation}
 and
 \begin{equation}\label{e2.33}
\displaystyle\int_{\mathbb{R}^N}b(x)|\Psi_n(x,0)|^{2_\alpha^*}dx=\displaystyle\int_{\mathbb{R}^N}b(x)|\omega_n(x,0)|^{2_\alpha^*}dx-\displaystyle\int_{\mathbb{R}^N}b(x)|\omega_\lambda^+|^{2_\alpha^*}dx+o_n(1).
 \end{equation}
From (\ref{e2.27}) and (\ref{e2.31})-(\ref{e2.33})  we deduce that
\begin{equation}\label{ee2.34}
\frac{1}{2}\|\Psi_n\|_X^2-\frac{1}{2_\alpha^*}\displaystyle\int_{\mathbb{R}^N}b(x)|\Psi_n(x,0)|^{2_\alpha^*}dx=\alpha_\lambda^+-I_\lambda(\omega_\lambda^+)+o_n(1).
 \end{equation}

As we discussed for (\ref{e2.31}), there holds
  \begin{equation}\label{eee2.35}
\displaystyle\int_{\mathbb{R}^N}a(x)\omega_n(x,0)\omega_\lambda^+(x,0)\ln|\omega_n(x,0)|dx\rightarrow
\displaystyle\int_{\mathbb{R}^N}a(x)|\omega_\lambda^+(x,0)|^2\ln|\omega_\lambda^+(x,0)|dx,
 \end{equation}
 as $n\rightarrow\infty$. Combining (\ref{ss2.42}) and (\ref{eee2.35}), we have
 \begin{equation*}
I'_\lambda(\omega_\lambda^+)\omega_\lambda^+=0,\ \ \text{i.e.} \ \ \omega_\lambda^+\in N_\lambda\cup\{0\}.
 \end{equation*}

Note that
  \begin{align}
(2_\alpha^*-2)\displaystyle\int_{\mathbb{R}^N}b(x)|\omega_\lambda^+(x,0)|^{2_\alpha^*}dx
\leq\displaystyle\liminf_{n\rightarrow\infty}(2_\alpha^*-2)
\displaystyle\int_{\mathbb{R}^N}b(x)|\omega_n(x,0)|^{2_\alpha^*}dx.\notag
 \end{align}
It follows from (\ref{e2.27}) that
    \begin{equation}\label{eee2.36}
\omega_\lambda^+\in N_\lambda^+.
 \end{equation}
According to Proposition \ref{p2.1}, it follows (\ref{ss2.42})-(\ref{e2.27}) and (\ref{e2.31})-(\ref{eee2.35}) that
 \begin{align}\label{ee2.36}
o_n(1)=I'_\lambda(\omega_n)\Psi_n&=(I'_\lambda(\omega_n)-I'_\lambda(\omega_\lambda^+))\Psi_n=\|\Psi_n\|_X^2-\displaystyle\int_{\mathbb{R}^N}b(x)|\Psi_n(x,0)|^{2_\alpha^*}dx,
\end{align}
as $n\rightarrow\infty$.

 We suppose that
\begin{equation*}
\|\Psi_n\|_X^2\rightarrow l~\text{and}~\displaystyle\int_{\mathbb{R}^N}b(x)|\Psi_n(x,0)|^{2_\alpha^*}dx\rightarrow l,
\ \text{as} \ n\rightarrow \infty,
 \end{equation*}
 for some $l\in[0,+\infty)$.

 If $l=0$, we obtain the desired result immediately. If $l>0$, we have $l\geq Sl^\frac{2}{2_\alpha^*}$ by Proposition \ref{p2.2}, and then
 \begin{equation}\label{sss2.36}
l\geq S^\frac{N}{2\alpha}.
 \end{equation}
It follows from (\ref{ee2.34}) and (\ref{eee2.36})-(\ref{sss2.36}) that
 \begin{equation*}
\alpha_\lambda^+=I_\lambda(\omega_\lambda^+)+\frac{l}{2}-\frac{l}{2_\alpha^*}\geq \alpha_\lambda^++\frac{\alpha}{N}l\geq \alpha_\lambda^++\frac{\alpha}{N}S^\frac{N}{2\alpha}.
 \end{equation*}
This is a contradiction. Hence, the only choice is $l=0$, i.e., $\omega_n\rightarrow\omega_\lambda^+$ in $X^\alpha(\mathbb{R}_+^{N+1})$ as $n\rightarrow\infty$.

 \textbf{Step 2.} We show that $\omega_\lambda^+(x,0)$ is a positive ground state solution of equation (\ref{e1.1}).

 Since $\omega_\lambda^+\in X^\alpha(\mathbb{R}_+^{N+1})$ is a local minimizer for $N_\lambda$. Lemma \ref{l2.2} tells us that $\omega_\lambda^+$ is a nontrivial solution of (\ref{e2.2}), and so $\omega_\lambda^+(x,0)$ is a nontrivial solution of equation (\ref{e1.1}). Note that $I_\lambda(|\omega_\lambda^+|)=\alpha_\lambda^+$. So we assume $\omega_\lambda^+(x,0)\geq0$. By virtue of the Maximum Principle for fractional elliptic equations \cite{13}, $\omega_\lambda^+$ is positive. Consequently, $\omega_\lambda^+(x,0)$ is a positive ground state solution of equation  (\ref{e1.1}).
\end{proof}

\begin{corollary}\label{c2.1}
\begin{description}
  \item[(i)] $I_\lambda(\omega_\lambda^+)\rightarrow0$, as $\lambda\rightarrow0$.
\item[(ii)] $\|\omega_\lambda^+\|_X\rightarrow0$, as $\lambda\rightarrow0$.
\end{description}
\end{corollary}

\begin{proof}
{\bf (i)} From Lemma \ref{l2.5} $(i)$ and (\ref{e2.23}), we have
\begin{equation*}
0>I_\lambda(\omega_\lambda^+)=\alpha_\lambda^+>-\lambda C.
 \end{equation*}
This implies $I_\lambda(\omega_\lambda^+)\rightarrow0$, as $\lambda\rightarrow0$.

{\bf (ii)} From (\ref{eee2.36}),  we have
\begin{equation*}
\displaystyle\int_{\mathbb{R}^N}a(x)|\omega_\lambda^+(x,0)|^2dx<0
 \end{equation*}
  and
 \begin{equation}\label{e2.37}
\displaystyle\int_{\mathbb{R}^N}b(x)|\omega_\lambda^+(x,0)|^{2_\alpha^*}dx<\frac{-\lambda}{2_\alpha^*-2}\displaystyle\int_{\mathbb{R}^N}a(x)|\omega_\lambda^+(x,0)|^2dx.
 \end{equation}
Similar to the derivation of (\ref{e2.12}),  for sufficiently small $\lambda>0$ we have
 \begin{equation}\label{e2.38}
\|\omega_\lambda^+\|_X\leq1.
 \end{equation}
It follows from (\ref{e2.6}) and Proposition \ref{p2.1} that
  \begin{align}\label{e2.39}
\displaystyle\int_{\mathbb{R}^N}a(x)|\omega_\lambda^+(x,0)|^2\ln|\omega_\lambda^+(x,0)|dx&\leq C(\|\omega_\lambda^+\|_X^2+|\omega_\lambda^+(x,0)|_2^4)\nonumber\\&\leq C(\|\omega_\lambda^+\|_X^2+\|\omega_\lambda^+\|_X^4).
\end{align}

From (\ref{e2.37})-(\ref{e2.39}) and Proposition \ref{p2.1}, we obtain
\begin{align}
\|\omega_\lambda^+\|_X^2&=\lambda\displaystyle\int_{\mathbb{R}^N}a(x)|\omega_\lambda^+(x,0)|^2\ln|\omega_\lambda^+(x,0)|dx+\displaystyle\int_{\mathbb{R}^N}b(x)|\omega_\lambda^+(x,0)|^{2_\alpha^*}dx\notag
\\ &\leq\lambda\displaystyle\int_{\mathbb{R}^N}a(x)|\omega_\lambda^+(x,0)|^2\ln|\omega_\lambda^+(x,0)|dx+\frac{-\lambda}{2_\alpha^*-2}\displaystyle\int_{\mathbb{R}^N}a(x)|\omega_\lambda^+(x,0)|^2dx\notag
\\&\leq\lambda C(\|\omega_\lambda^+\|_X^2+\|\omega_\lambda^+\|_X^4)+\lambda\|\omega_\lambda^+\|_X^2\notag
\\&\leq\lambda C.\notag
\end{align}
\end{proof}

\section{Proof of Theorem \ref{t1.2}}

In this section, we present some results on compactness and estimates,  and then prove Theorem \ref{t1.2}.
Throughout this section, we always suppose that conditions $(H_1)-(H_3)$ hold.

\begin{lemma}\label{l3.1}
The following two statements are true.
\begin{description}
  \item[(i)] If  $a(x)$ is negative or sign-changing, then $I_\lambda$ satisfies the $(PS)_c$ condition for $c\in \left(-\infty,\, \alpha_\lambda^++\frac{\alpha}{N}S^\frac{N}{2\alpha} \right)$.
  \item[(ii)] If $a(x)\geq0$, then $I_\lambda$ satisfies the $(PS)_c$ condition for $c\in \left(-\infty,\, \frac{\alpha}{N}S^\frac{N}{2\alpha}\right)$.
\end{description}
\end{lemma}

\begin{proof}
Let $\{\omega_n\}\subset X^\alpha(\mathbb{R}_+^{N+1})$ be a $(PS)_c$ sequence for $I_\lambda$.
We claim that there exists some $C>0$ such that
\begin{equation}\label{e3.1}
\|\omega_n\|_X\leq C,\ \, n\in\mathbb{Z}_+.
 \end{equation}

Suppose otherwise that $\|\omega_n\|_X\rightarrow\infty$, as $n\rightarrow\infty$. A straightforward calculation gives
\begin{align}
2\lambda\|\omega_n\|_X&\geq1+\frac{\alpha}{N}S^\frac{N}{2\alpha}+\lambda\|\omega_n\|_X\notag
\\&\geq I_\lambda(\omega_n)-\frac{1}{2}I'_\lambda(\omega_n)\omega_n\notag
\\&\geq\frac{\lambda}{4}\displaystyle\int_{\mathbb{R}^N}a(x)|\omega_n(x,0)|^2dx+\left(\frac{1}{2}-\frac{1}{2_\alpha^*}\right)\displaystyle\int_{\mathbb{R}^N}b(x)|\omega_n(x,0)|^{2_\alpha^*}dx\notag
\\&\geq\frac{\lambda}{4}\displaystyle\int_{\mathbb{R}^N}a(x)|\omega_n(x,0)|^2dx.\notag
\end{align}
For sufficiently large $n$,  there holds
 \begin{equation}\label{e3.2}
\displaystyle\int_{\mathbb{R}^N}a(x)|\omega_n(x,0)|^2dx\leq C\|\omega_n\|_X
 \end{equation}
 for some $C>0$.

It follows Proposition \ref{p2.1} that
 \begin{equation}\label{e3.3}
 \ln |\omega_n(x,0)|_2^2=2 \ln |\omega_n(x,0)|_2\leq 2|\omega_n(x,0)|_2\leq C\|\omega_n\|_X
 \end{equation}
 for some $C>0$. Using (\ref{e2.6}) and (\ref{e3.2})-(\ref{e3.3}) yields
  \begin{equation*}\label{e3.4}
\displaystyle\int_{\mathbb{R}^N}a(x)|\omega_n(x,0)|^2\ln|\omega_n(x,0)|dx\leq C\|\omega_n\|^2_X.
 \end{equation*}
Thus, we have
  \begin{align}
&1+\frac{\alpha}{N}S^\frac{N}{2\alpha}+\lambda\|\omega_n\|_X\notag
\\&\geq I_\lambda(\omega_n)-\frac{1}{2_\alpha^*}I'_\lambda(\omega_n)\omega_n\notag
\\&=\left(\frac{1}{2}-\frac{1}{2_\alpha^*}\right)\|\omega_n\|_X^2-\lambda\left(\frac{1}{2}-\frac{1}{2_\alpha^*}\right)\displaystyle\int_{\mathbb{R}^N}a(x)|\omega_n(x,0)|^2\ln|\omega_n(x,0)|dx\notag
\\&~~~~+\frac{\lambda}{4}\displaystyle\int_{\mathbb{R}^N}a(x)|\omega_n(x,0)|^2dx\notag
\\&\geq\left(\frac{1}{2}-\frac{1}{2_\alpha^*}\right)\|\omega_n\|_X^2-\lambda C\|\omega_n\|_X^2\notag
\\&\geq C\|\omega_n\|_X^2\notag
\end{align}
for sufficiently small $\lambda>0$. This obviously contradicts our assumption that $\{\|\omega_n\|_X\}$ is unbounded. Hence,  (\ref{e3.1}) holds.

From (\ref{e3.1}), there exists a subsequence (still denoted by $\{\omega_n\}$) and $\omega_0\in X^\alpha \left(\mathbb{R}_+^{N+1}\right)$ such that
 \begin{equation}\label{ee3.4}
\omega_n\rightharpoonup \omega_0~\text{in}~X^\alpha(\mathbb{R}_+^{N+1}), \ \text{as}\ n\rightarrow\infty.
 \end{equation}
Similar to the derivations of (\ref{e2.27}) and (\ref{e2.31}),  as $n\rightarrow\infty$ we have
\begin{equation}\label{e3.5}
\displaystyle\int_{\mathbb{R}^N}a(x)|\omega_n(x,0)|^2dx\rightarrow\displaystyle\int_{\mathbb{R}^N}a(x)|\omega_0(x,0)|^2dx
 \end{equation}
 and
 \begin{equation}\label{e3.6}
\displaystyle\int_{\mathbb{R}^N}a(x)|\omega_n(x,0)|^2\ln|\omega_n(x,0)|dx
\rightarrow\displaystyle\int_{\mathbb{R}^N}a(x)|\omega_0(x,0)|^2\ln|\omega_0(x,0)|dx.
 \end{equation}

 Set $\Psi_n=\omega_n-\omega_0$. It follows from Brezis-Lieb's lemma \cite{32} that
 \begin{equation}\label{e3.7}
\|\Psi_n\|_X^2=\|\omega_n\|_X^2-\|\omega_0\|_X^2+o_n(1)
 \end{equation}
 and
 \begin{equation}\label{e3.8}
\displaystyle\int_{\mathbb{R}^N}b(x)|\Psi_n(x,0)|^{2_\alpha^*}dx=\displaystyle\int_{\mathbb{R}^N}b(x)|\omega_n(x,0)|^{2_\alpha^*}dx-\displaystyle\int_{\mathbb{R}^N}b(x)|\omega_0|^{2_\alpha^*}dx+o_n(1).
 \end{equation}
From (\ref{e3.5})-(\ref{e3.8}) we deduce
\begin{equation}\label{e3.9}
\frac{1}{2}\|\Psi_n\|_X^2-\frac{1}{2_\alpha^*}\displaystyle\int_{\mathbb{R}^N}b(x)|\Psi_n(x,0)|^{2_\alpha^*}dx=c-I_\lambda(\omega_0)+o_n(1).
 \end{equation}

Similar to the derivation of  (\ref{e2.31}),  as $n\rightarrow\infty$ we have
  \begin{equation}\label{e3.10}
\displaystyle\int_{\mathbb{R}^N}a(x)\omega_n(x,0)\omega_0(x,0)\ln|\omega_n(x,0)|dx\rightarrow\displaystyle\int_{\mathbb{R}^N}a(x)|\omega_0(x,0)|^2\ln|\omega_0(x,0)|dx,
 \end{equation}
 which together with (\ref{ee3.4}) leads to
$I'_\lambda(\omega_0)\omega_0=0,$ i.e. $\omega_0\in N_\lambda\cup\{0\}.$

By (\ref{ee3.4})-(\ref{e3.10}) and Proposition \ref{p2.1}, as $n\rightarrow\infty$ we obtain
 \begin{align}\label{e3.11}
o_n(1)=I'_\lambda(\omega_n)\Psi_n&=(I'_\lambda(\omega_n)-I'_\lambda(\omega_0))\Psi_n=\|\Psi_n\|_X^2-\displaystyle\int_{\mathbb{R}^N}b(x)|\Psi_n(x,0)|^{2_\alpha^*}dx.
\end{align}
We may suppose that
\begin{equation*}
\|\Psi_n\|_X^2\rightarrow l~\text{and}~\displaystyle\int_{\mathbb{R}^N}b(x)|\Psi_n(x,0)|^{2_\alpha^*}dx\rightarrow l
 \end{equation*}
 for some $l\in[0,+\infty)$.

 If $l=0$,  we can obtain $\omega_n\rightarrow\omega_0$ as $n\rightarrow\infty$ immediately. If $l>0$, we divide our discussions into two cases.

$\textbf{(i)}$. If  $a(x)$ is negative or sign-changing, it follows from Lemma \ref{l2.5} $(i)$ that $\alpha_\lambda^-\geq \alpha_\lambda^+$. From Proposition \ref{p2.2}, we have $l\geq Sl^\frac{2}{2_\alpha^*}$. In view of (\ref{e3.9}), (\ref{e3.11}) and $\omega_0\in N_\lambda\cup\{0\}$, we get
 \begin{equation*}
\frac{\alpha}{N}S^\frac{N}{2\alpha}+\alpha_\lambda^+>c=I_\lambda(\omega_0)+\frac{l}{2}-\frac{l}{2_\alpha^*}\geq \frac{\alpha}{N}l+\alpha_\lambda^+\geq \frac{\alpha}{N}S^\frac{N}{2\alpha}+\alpha_\lambda^+.
 \end{equation*}
 This is a contradiction. Hence, $l=0$, i.e. $\omega_n\rightarrow\omega_0$ in $X^\alpha(\mathbb{R}_+^{N+1})$ as  $n\rightarrow\infty$.

$\textbf{(ii)}$. If $a(x)\geq0$, it follows from Lemma \ref{l2.5} $(ii)$ that $N_\lambda^+=\emptyset$ and $\alpha_\lambda^->0$. From Proposition \ref{p2.2}, we get $l\geq Sl^\frac{2}{2_\alpha^*}$. In view of (\ref{e3.9}), (\ref{e3.11}) and $\omega_0\in N^-_\lambda\cup\{0\}$ we have
 \begin{equation*}
\frac{\alpha}{N}S^\frac{N}{2\alpha}>c=I_\lambda(\omega_0)+\frac{l}{2}-\frac{l}{2_\alpha^*}\geq \frac{\alpha}{N}l\geq \frac{\alpha}{N}S^\frac{N}{2\alpha},
 \end{equation*}
 which yields another contradiction. Hence, $l=0$, i.e. $\omega_n\rightarrow\omega_0$ in $X^\alpha(\mathbb{R}_+^{N+1})$ as  $n\rightarrow\infty$.
\end{proof}

Let $\eta(x,y)\in C^\infty(\mathbb{R}^{N}\times\mathbb{R})$ such that $0\leq\eta\leq1$, $|\nabla\eta|\leq C$ and
\begin{displaymath}
\eta(x,y):=\left\{\begin{array}{ll}
1,~ (x,y)\in B_\frac{\delta_0}{2}^+:=\{(x,y)\in \mathbb{R}_+^{N+1};\sqrt{|x_1|^2+|x_2|^2+\cdots+|x_N|^2+|y|^2}<\frac{r_0}{2},y>0\},\\
0, ~ (x,y)\not\in B_{\delta_0}^+:=\{(x,y)\in \mathbb{R}_+^{N+1};\sqrt{|x_1|^2+|x_2|^2+\cdots+|x_N|^2+|y|^2}<r_0,y>0\},
\end{array}\right.
\end{displaymath}
where $\delta_0<r_0$ and $r_0$ is defined in Remark \ref{r1.1}.

Set
\begin{equation}\label{e3.13}
 v_{\varepsilon,z}=\eta(x-z,y)\omega_\varepsilon(x-z,y),~~z\in M.
\end{equation}
Following \cite{17}, we deduce that
\begin{equation}\label{e3.14}
 \displaystyle\int_{\mathbb{R}_+^{N+1}}y^{1-2\alpha}|\nabla v_{\varepsilon,z}|^2dxdy=\displaystyle\int_{\mathbb{R}_+^{N+1}}y^{1-2\alpha}|\nabla \omega_\varepsilon|^2dxdy+O(\varepsilon^{N-2\alpha}),
\end{equation}
\begin{equation}\label{e3.15}
\displaystyle\int_{\mathbb{R}^N}|v_{\varepsilon,z}(x,0)|^qdx=\left\{\begin{array}{ll}
O(\varepsilon^\frac{2N-(N-2\alpha)q}{2}),~ &\text{if}~q>\frac{N}{N-2\alpha},\\
O(\varepsilon^\frac{(N-2\alpha)q}{2}),~~~~ &\text{if}~q\leq\frac{N}{N-2\alpha},
\end{array}\right.
\end{equation}
and
\begin{equation}\label{e3.16}
\displaystyle\int_{\mathbb{R}^N}|v_{\varepsilon,z}(x,0)|^2dx=\left\{\begin{array}{ll}
O(\varepsilon^{2\alpha}),~ &\text{if}~N>4\alpha,\\
O(\varepsilon^{2\alpha}\ln(\frac{1}{\varepsilon})),~ &\text{if}~N=4\alpha,\\
O(\varepsilon^{N-2\alpha}),~~~~~ &\text{if}~N<4\alpha.
\end{array}\right.
\end{equation}

\begin{lemma}\label{l3.2}
$
\displaystyle\int_{\mathbb{R}^N}b(x) |v_{\varepsilon,z}(x,0)|^{2_\alpha^*}dx=\displaystyle\int_{\mathbb{R}^N}|\omega_\varepsilon(x,0)|^{2_\alpha^*}dx+
O(\varepsilon^N).
$
\end{lemma}
\begin{proof}
By Remark \ref{r1.1} and the definition of $v_{\varepsilon,z}$, we have
\begin{align}
0&\leq\frac{1}{\varepsilon^N}\left[\displaystyle\int_{\mathbb{R}^N}|\omega_\varepsilon(x,0)|^{2_\alpha^*}dx-\displaystyle\int_{\mathbb{R}^N}b(x) |v_{\varepsilon,z}(x,0)|^{2_\alpha^*}dx\right]\notag
\\&=C\displaystyle\int_{\mathbb{R}^N\backslash B_\frac{\delta_0}{2}}\frac{b(z)-b(x+z)\eta^{2_\alpha^*}(x,0)}{(\varepsilon^2+|x|^2)^N}dx+C\displaystyle\int_{ B_\frac{\delta_0}{2}}\frac{b(z)-b(x+z)}{(\varepsilon^2+|x|^2)^N}dx\notag
\\&\leq C\displaystyle\int_{\mathbb{R}^N\backslash B_\frac{\delta_0}{2}}\frac{1}{|x|^{2N}}dx+C\displaystyle\int_{B_\frac{\delta_0}{2}}\frac{|x|^\rho}{(\varepsilon^2+|x|^2)^N}dx\notag
\\&\leq C\displaystyle\int_\frac{\delta_0}{2}^{+\infty}r^{-(N+1)}dr+C\displaystyle\int_0^\frac{\delta_0}{2}r^{\rho-N-1}dr
\nonumber\\&\leq C\notag
\end{align}
for $z\in M$.
\end{proof}

Set
\begin{equation*}
I_\infty(\omega):=\frac{1}{2}\|\omega\|_X^2-\frac{1}{2_\alpha^*}\displaystyle\int_{\mathbb{R}^N}b(x)|\omega(x,0)|^{2_\alpha^*}dx
\end{equation*}
and
\begin{equation*}
I^\infty(\omega):=\frac{1}{2}\|\omega\|_X^2-\frac{1}{2_\alpha^*}\displaystyle\int_{\mathbb{R}^N}|\omega(x,0)|^{2_\alpha^*}dx.
\end{equation*}

Define the Nehari manifold associated to $I_\infty$ and $I^\infty$ by
\begin{equation*}
N_\infty(\omega):=\{\omega\in X^\alpha(\mathbb{R}_+^{N+1})\backslash\{0\};\ (I_\infty)'(\omega)\omega=0\}.
\end{equation*}
and
\begin{equation*}
N^\infty(\omega):=\{\omega\in X^\alpha(\mathbb{R}_+^{N+1})\backslash\{0\};\ (I^\infty)'(\omega)\omega=0\}.
\end{equation*}

\begin{lemma}\label{l3.20}
$
\displaystyle\inf_{\omega\in N^\infty} I^\infty(\omega)=\displaystyle\inf_{\omega\in N_\infty} I_\infty(\omega)=\frac{\alpha}{N}S^\frac{N}{2\alpha}.
$
\end{lemma}

\begin{proof}
Since
\begin{equation}\label{ss3.16}
\displaystyle\max_{t\geq0}\left(\frac{a}{2}t^2-\frac{b}{2_\alpha^*}t^{2_\alpha^*}\right)=\frac{\alpha}{N}\left(\frac{a^2}{b^\frac{2}{2_\alpha^*}}\right)^\frac{2_\alpha^*}{2_\alpha^*-2}
\end{equation}
for any $a,b>0$. It follows from Proposition \ref{p2.2} that
\begin{align}\label{s3.17}
\displaystyle\inf_{\omega\in N^\infty} I^\infty(\omega)&=\displaystyle\inf_{\omega\in X^\alpha(\mathbb{R}_+^{N+1}) \backslash\{0\}}\sup_{t\geq0}I^\infty(t\omega)\notag
\\&=\displaystyle\inf_{\omega\in X^\alpha(\mathbb{R}_+^{N+1}) \backslash\{0\}}\frac{\alpha}{N}\left(\frac{\|\omega\|_X^2}{\left(\displaystyle\int_{\mathbb{R}^N} |\omega(x,0)|^{2_\alpha^*}dx\right)^\frac{2}{2_\alpha^*}}\right)^\frac{N}{2\alpha}\notag
\\&\geq \frac{\alpha}{N}S^\frac{N}{2\alpha}.
\end{align}
On the other hand,  from  (\ref{e3.14}), (\ref{e3.16}) and Lemma  \ref{l3.2} it follows that
\begin{align}
\displaystyle\sup_{t\geq0}I_\infty(tv_{\varepsilon,z})&=\frac{\alpha}{N}\left(\frac{\|v_{\varepsilon,z}\|_{X^\alpha}^2}{\left(\displaystyle\int_{\mathbb{R}^N} b(x)|v_{\varepsilon,z}(x,0)|^{2_\alpha^*}dx\right)^\frac{2}{{2_\alpha^*}}}\right)^\frac{N}{2\alpha}\notag
\\&=\frac{\alpha}{N}S^\frac{N}{2\alpha}+o_\varepsilon(1), \ \text{as}\ \varepsilon\rightarrow 0. \notag
\end{align}
That is,
\begin{equation}\label{s3.18}
\displaystyle\inf_{\omega\in N_\infty} I_\infty(\omega)\leq\frac{\alpha}{N}S^\frac{N}{2\alpha}.
\end{equation}

Note that $b(x)\leq1$.   From (\ref{s3.17}) and (\ref{s3.18}) we obtain
\begin{align}
\frac{\alpha}{N}S^\frac{N}{2\alpha}&\leq\displaystyle\inf_{\omega\in N^\infty} I^\infty(\omega)\nonumber\\&=\displaystyle\inf_{\omega\in X^\alpha(\mathbb{R}_+^{N+1}) \backslash\{0\}}\sup_{t\geq0}I^\infty(t\omega)\notag
\\&\leq\displaystyle\inf_{\omega\in X^\alpha(\mathbb{R}_+^{N+1}) \backslash\{0\}}\sup_{t\geq0} I_\infty(t\omega)\notag
\\& =\displaystyle\inf_{\omega\in N_\infty} I_\infty(\omega)\nonumber\\&\leq\frac{\alpha}{N}S^\frac{N}{2\alpha}.\notag
\end{align}
\end{proof}

\begin{lemma}\label{l3.3}
The following two statements are true.
\begin{description}
  \item[(i)] If $a(x)\geq 0$, then there exists $\varepsilon_0>0$ small enough  such that for $\varepsilon\in(0,\varepsilon_0)$ we have
  \begin{equation*}
\displaystyle\max_{t\geq0}I_\lambda(tv_{\varepsilon,z})<\frac{\alpha}{N}S^\frac{N}{2\alpha}-\sigma(\varepsilon_0)
\end{equation*}
uniformly with respect to $z\in M$, where $\sigma(\varepsilon_0)$ is a small positive constant. Furthermore, there exists $t_{\varepsilon,z}^->0$ such that
\begin{equation*}
t_{\varepsilon,z}^-v_{\varepsilon,z}\in N_\lambda^-.
\end{equation*}
  \item[(ii)] If $a(x)$ is sign-changing, then there exists $\overline{\varepsilon}_0>0$ small enough  such that for $\varepsilon\in(0,\overline{\varepsilon}_0)$ we have
  \begin{equation*}
\displaystyle\max_{t\geq0}I_\lambda(\omega_\lambda^++tv_{\varepsilon,z})<\alpha_\lambda^++\frac{\alpha}{N}S^\frac{N}{2\alpha}-\overline{\sigma}(\overline{\varepsilon}_0)
\end{equation*}
uniformly with respect to $z\in M$, where $\overline{\sigma}(\overline{\varepsilon}_0)$ is a small positive constant.
\end{description}
\end{lemma}

\begin{proof}
$\textbf{(i).} $ Since
\begin{equation*}
\displaystyle\lim_{t\rightarrow0^+}I_\lambda(tv_{\varepsilon,z})=0~\text{and}~\displaystyle\lim_{t\rightarrow+\infty}I_\lambda(tv_{\varepsilon,z})=-\infty
\end{equation*}
for sufficiently small $\varepsilon>0$ and $z\in M$, there exist small $t_1>0$  and large $t_2>0$  such that
\begin{equation}\label{e3.17}
 I_\lambda(tv_{\varepsilon,z})<\frac{\alpha}{N}S^\frac{N}{2\alpha},~~t\in (0,t_1]\cup[t_2,+\infty).
\end{equation}

For $t\in [t_1,t_2]$, using (\ref{e3.13})-(\ref{e3.14}) and Lemma \ref{l3.2} we have
\begin{align}
I_\lambda(tv_{\varepsilon,z})&\leq \displaystyle\max_{t\geq0}\left(\frac{t^2}{2}\|v_{\varepsilon,z}\|_X^2-\frac{t^{2_\alpha^*}}{2_\alpha^*}\displaystyle\int_{\mathbb{R}^N}b(x) |v_{\varepsilon,z}(x,0)|^{2_\alpha^*}dx\right)\notag
\\& ~~~~+\lambda C\displaystyle\int_{\mathbb{R}^N}a(x)|v_{\varepsilon,z}(x,0)|^2dx-\lambda\frac{t_1^2}{2}\displaystyle\int_{\mathbb{R}^N}a(x)|v_{\varepsilon,z}(x,0)|^2\ln|v_{\varepsilon,z}(x,0)|dx\notag
\\&\leq\frac{\alpha}{N}\left(\frac{\|v_{\varepsilon,z}\|_X^2}{\left(\displaystyle\int_{\mathbb{R}^N}b(x) |v_{\varepsilon,z}(x,0)|^{2_\alpha^*}dx\right)^\frac{2}{2_\alpha^*}}\right)^\frac{N}{2\alpha}+\lambda C\displaystyle\int_{\mathbb{R}^N}|v_{\varepsilon,z}(x,0)|^2dx\notag
\\&~~~~-\lambda C\displaystyle\int_{\mathbb{R}^N}|v_{\varepsilon,z}(x,0)|^2\ln|v_{\varepsilon,z}(x,0)|dx\notag
\\&\leq\frac{\alpha}{N}S^\frac{N}{2\alpha}+O(\epsilon^{N-2\alpha})+\lambda C \displaystyle\int_{\mathbb{R}^N}|v_{\epsilon,z}(x,0)|^2dx\notag
\\&~~~~-\lambda C\displaystyle\int_{\mathbb{R}^N}|v_{\varepsilon,z}(x,0)|^2\ln|v_{\varepsilon,z}(x,0)|dx,\notag
\end{align}
which together with (\ref{e3.16}) and
\begin{align}
&~~~~\displaystyle\int_{\mathbb{R}^N}|v_{\varepsilon,z}(x,0)|^2\ln|v_{\varepsilon,z}(x,0)|dx\notag
\\&=\displaystyle\int_{B_\varepsilon(z)}|v_{\varepsilon,z}(x,0)|^2\ln|v_{\varepsilon,z}(x,0)|dx+\displaystyle\int_{B_{r_0}(z)\backslash B_\varepsilon(z)}|v_{\varepsilon,z}(x,0)|^2\ln|v_{\varepsilon,z}(x,0)|dx\notag
\\&\geq C\displaystyle\int_{B_\varepsilon}\varepsilon^{-(N-2\alpha)}\ln\varepsilon^{-\frac{N-2\alpha}{2}}dx+C\epsilon^{N-2\alpha}\displaystyle\int_{\{\varepsilon<|x|<r_0\}}\frac{1}{(|x|^2)^{N-2\alpha}}\ln\frac{C\varepsilon^\frac{N-2\alpha}{2}}{|x|^{N-2\alpha}}dx\notag
\\&\geq C\varepsilon^{2\alpha}\ln\varepsilon^{-1}+C\epsilon^{N-2\alpha}\displaystyle\int_{\{\varepsilon<|x|<r_0\}}\frac{1}{(|x|^2)^{N-2\alpha}}\ln C\varepsilon dx\notag
\\&=C\varepsilon^{2\alpha}\ln\varepsilon^{-1}+C\varepsilon^{N-2\alpha}\ln C\varepsilon,\notag
\end{align}
show that
\begin{align}\label{e3.19}
 I_\lambda(tv_{\varepsilon,z})\leq\frac{\alpha}{N}S^\frac{N}{2\alpha}+O(\varepsilon^{N-2\alpha})+O(\varepsilon^{2\alpha})-C\varepsilon^{2\alpha}\ln\varepsilon^{-1}+C\varepsilon^{N-2\alpha}\ln C\varepsilon<\frac{\alpha}{N}S^\frac{N}{2\alpha},
\end{align}
for $\varepsilon>0$ small and $z\in M$, where we have used the assumption $N>4\alpha$.

It follows from (\ref{e3.17}) and (\ref{e3.19}) that there exists $\varepsilon_0>0$ small enough  such that for $\varepsilon\in(0,\varepsilon_0)$ there holds
  \begin{equation*}
\displaystyle\max_{t\geq0}I_\lambda(tv_{\varepsilon,z})<\frac{\alpha}{N}S^\frac{N}{2\alpha}-\sigma(\varepsilon_0)
\end{equation*}
uniformly with respect to $z\in M$, where $\sigma(\varepsilon_0)$ is a small positive constant. Meanwhile, it follows from Lemma \ref{l2.1} $(i)$ that there exists $t_{\varepsilon,z}^->0$ such that
\begin{equation*}
t_{\varepsilon,z}^-v_{\varepsilon,z}\in N_\lambda^-.
\end{equation*}

$\textbf{(ii)}$.  Since
\begin{equation*}
\displaystyle\lim_{t\rightarrow0^+}I_\lambda(\omega_{\lambda}^++tv_{\varepsilon,z})
=\alpha_{\lambda}^+~\text{and}~\displaystyle\lim_{t\rightarrow+\infty}
I_\lambda(\omega_{\lambda}^++tv_{\varepsilon,z})=-\infty
\end{equation*}
for small $\varepsilon>0$ and $z\in M$,  there exist small $t_1>0$  and large $t_2>0$ such that
\begin{equation}\label{e3.23}
 I_\lambda(\omega_{\lambda}^++tv_{\varepsilon,z})<\alpha_{\lambda}^++\frac{\alpha}{N}S^\frac{N}{2\alpha},
 ~~t\in (0,t_1]\cup[t_2,+\infty).
\end{equation}

For $t\in [t_1,t_2]$, taking into account $I'_\lambda(\omega_{\lambda}^+)v_{\varepsilon,z}=0$ and $I_\lambda(\omega_{\lambda}^+)=\alpha_{\lambda}^+$ we derive that
\begin{align}\label{e3.24}
 &I_\lambda(\omega_{\lambda}^++tv_{\varepsilon,z})\notag
\\ &=\frac{1}{2}\|\omega_{\lambda}^+\|_X^2+\frac{t^2}{2}\|v_{\varepsilon,z}\|_X^2+t\displaystyle\int_{\mathbb{R}_+^{N+1}}y^{1-2\alpha}\nabla\omega_{\lambda}^+\nabla v_{\varepsilon,z}dxdy+t\displaystyle\int_{\mathbb{R}^N}\omega_{\lambda}^+(x,0)v_{\varepsilon,z}(x,0)dx\notag
\\&~~~~-\frac{\lambda}{2}\displaystyle\int_{\mathbb{R}^N}a(x)|\omega_{\lambda}^+(x,0)+tv_{\varepsilon,z}(x,0)|^2\ln|\omega_{\lambda}^+(x,0)+tv_{\varepsilon,z}(x,0)|dx\notag
\\&~~~~+\frac{\lambda}{4}\displaystyle\int_{\mathbb{R}^N}a(x)|\omega_{\lambda}^+(x,0)|^2dx+\frac{\lambda t^2}{4}\displaystyle\int_{\mathbb{R}^N}a(x)|v_{\varepsilon,z}(x,0)|^2dx+\frac{\lambda t}{2}\displaystyle\int_{\mathbb{R}^N}a(x)\omega_{\lambda}^+v_{\varepsilon,z}(x,0)dx\notag
\\&~~~~-\frac{1}{2_\alpha^*}\displaystyle\int_{\mathbb{R}^N}b(x) |\omega_{\lambda}^+(x,0)+tv_{\varepsilon,z}(x,0)|^{2_\alpha^*}dx\notag
\\&=I_\lambda(\omega_{\lambda}^+)+I_\lambda(tv_{\varepsilon,z})-\frac{\lambda}{2}\displaystyle\int_{\mathbb{R}^N}a(x)|\omega_{\lambda}^+(x,0)+tv_{\varepsilon,z}(x,0)|^2\ln|\omega_{\lambda}^+(x,0)+tv_{\varepsilon,z}(x,0)|dx\notag
\\&~~~~+\frac{\lambda}{2}\displaystyle\int_{\mathbb{R}^N}a(x)|\omega_{\lambda}^+(x,0)|^2\ln|\omega_{\lambda}^+(x,0)|dx+\frac{\lambda}{2}\displaystyle\int_{\mathbb{R}^N}a(x)|tv_{\varepsilon,z}(x,0)|^2\ln|tv_{\varepsilon,z}(x,0)|dx\notag
\\&~~~~+\lambda t\displaystyle\int_{\mathbb{R}^N}a(x)|\omega_{\lambda}^+(x,0)|v_{\varepsilon,z}(x,0)\ln|\omega_{\lambda}^+(x,0)|dx+\frac{\lambda t}{2}\displaystyle\int_{\mathbb{R}^N}a(x)\omega_{\lambda}^+(x,0)v_{\varepsilon,z}(x,0)dx\notag
\\&~~~~-\frac{1}{2_\alpha^*}\displaystyle\int_{\mathbb{R}^N}b(x) |\omega_{\lambda}^+(x,0)+tv_{\varepsilon,z}(x,0)|^{2_\alpha^*}dx\notag
\\&~~~~+\frac{1}{2_\alpha^*}\displaystyle\int_{\mathbb{R}^N}b(x) [|\omega_{\lambda}^+(x,0)|^{2_\alpha^*}+|tv_{\varepsilon,z}(x,0)|^{2_\alpha^*}+2_\alpha^*|\omega_{\lambda}^+(x,0)|^{2_\alpha^*-2}\omega_{\lambda}^+(x,0)tv_{\varepsilon,z}(x,0)]dx\notag
\\&=\alpha_{\lambda}^++I_\lambda(tv_{\varepsilon,z})-i(t)-j(t),
\end{align}
where
\begin{align}
 i(t):&=\frac{\lambda}{2}\displaystyle\int_{\mathbb{R}^N}a(x)|\omega_{\lambda}^+(x,0)+tv_{\varepsilon,z}(x,0)|^2\ln|\omega_{\lambda}^+(x,0)+tv_{\varepsilon,z}(x,0)|dx\notag
\\&~~~~-\frac{\lambda}{2}\displaystyle\int_{\mathbb{R}^N}a(x)|\omega_{\lambda}^+(x,0)|^2\ln|\omega_{\lambda}^+(x,0)|dx-\frac{\lambda}{2}\displaystyle\int_{\mathbb{R}^N}a(x)|tv_{\varepsilon,z}(x,0)|^2\ln|tv_{\varepsilon,z}(x,0)|dx\notag
\\&~~~~-\lambda \displaystyle\int_{\mathbb{R}^N}a(x)|\omega_{\lambda}^+(x,0)|tv_{\varepsilon,z}(x,0)\ln|\omega_{\lambda}^+(x,0)|dx-\frac{\lambda }{2}\displaystyle\int_{\mathbb{R}^N}a(x)\omega_{\lambda}^+(x,0)tv_{\varepsilon,z}(x,0)dx\notag
\end{align}
and
\begin{align}\label{ee3.24}
 j(t):&=\frac{1}{2_\alpha^*}\displaystyle\int_{\mathbb{R}^N}b(x) |\omega_{\lambda}^+(x,0)+tv_{\varepsilon,z}(x,0)|^{2_\alpha^*}dx
\\&~~~~-\frac{1}{2_\alpha^*}\displaystyle\int_{\mathbb{R}^N}b(x) [|\omega_{\lambda}^+(x,0)|^{2_\alpha^*}+|tv_{\varepsilon,z}(x,0)|^{2_\alpha^*}+2_\alpha^*|\omega_{\lambda}^+(x,0)|^{2_\alpha^*-2}\omega_{\lambda}^+(x,0)tv_{\varepsilon,z}(x,0)]dx.\notag
\end{align}

Let us estimate $i(t)$ and $j(t)$ separately. To estimate $i(t)$, we define
\begin{equation*}
f(t):=\ln(1+t)+2t\ln(1+t)-t,~~t\geq0.
\end{equation*}
Notice that
\begin{equation*}
f(0)=0~\text{and}~f'(t)=\frac{t}{1+t}+2\ln(1+t)\geq0,~~t\geq0.
\end{equation*}
So we have
\begin{equation*}
f(t)\geq0,~~t\geq0,
\end{equation*}
which implies that
\begin{align}\label{e3.25}
 i(t)&=\frac{\lambda}{2}\displaystyle\int_{\mathbb{R}^N}a(x)|\omega_{\lambda}^+(x,0)|^2\ln\left(1+\frac{tv_{\varepsilon,z}(x,0)}{\omega_{\lambda}^+(x,0)}\right)dx\notag
 \\&~~~~+\lambda\displaystyle\int_{\mathbb{R}^N}a(x)|\omega_{\lambda}^+(x,0)|^2\frac{tv_{\varepsilon,z}(x,0)}{\omega_{\lambda}^+(x,0)}\ln\left(1+\frac{tv_{\varepsilon,z}(x,0)}{\omega_{\lambda}^+(x,0)}\right)dx\notag
 \\&~~~~-\frac{\lambda }{2}\displaystyle\int_{\mathbb{R}^N}a(x)|\omega_{\lambda}^+(x,0)|^2\frac{tv_{\varepsilon,z}(x,0)}{\omega_{\lambda}^+(x,0)}dx\notag
 \\&~~~~+\frac{\lambda}{2}\displaystyle\int_{\mathbb{R}^N}a(x)|tv_{\varepsilon,z}(x,0)|^2\ln \left(1+\frac{\omega_{\lambda}^+(x,0)}{tv_{\varepsilon,z}(x,0)}\right)dx\notag
 \\&\geq\frac{\lambda}{2}\displaystyle\int_{\mathbb{R}^N}a(x)|tv_{\varepsilon,z}(x,0)|^2\ln \left(1+\frac{\omega_{\lambda}^+(x,0)}{tv_{\varepsilon,z}(x,0)}\right)dx \notag\\
 &\geq0.
\end{align}

To estimate $j(t)$, we follow \cite[formulas (17) and (21)]{35}  and  (\ref{e3.15}) to derive that
\begin{align}
&\displaystyle\int_{\mathbb{R}^N}b(x)|\omega_{\lambda}^+(x,0)+tv_{\varepsilon,z}(x,0)|^{2_\alpha^*}dx\notag
\\=& \displaystyle\int_{\mathbb{R}^N}b(x)|\omega_{\lambda}^+(x,0)|^{2_\alpha^*}dx+2_\alpha^*t\displaystyle\int_{\mathbb{R}^N}b(x)(\omega_{\lambda}^+(x,0))^{2_\alpha^*-1}v_{\varepsilon,z}(x,0)dx\notag
\\&+t^{2_\alpha^*}\displaystyle\int_{\mathbb{R}^N}b(x)|v_{\varepsilon,z}(x,0)|^{2_\alpha^*}dx\notag
\\&+2_\alpha^*t^{2_\alpha^*-1}\displaystyle\int_{\mathbb{R}^N}b(x)(v_{\varepsilon,z}(x,0))^{2_\alpha^*-1}\omega_{\lambda}^+(x,0)dx
+o\left(\varepsilon^{\frac{N-2\alpha}{2}}\right),\notag
\end{align}
Using this estimate together with (\ref{e3.15}) and (\ref{ee3.24}) leads to
\begin{align}\label{e3.26}
 j(t)&\geq |t_1|^{2_\alpha^*-1}\displaystyle\int_{\mathbb{R}^N}b(x)\omega_{\lambda}^+(x,0) |v_{\varepsilon,z}(x,0)|^{2_\alpha^*-1}dx+o(\varepsilon^{\frac{N-2\alpha}{2}})\notag
\\&\geq C\displaystyle\int_{\mathbb{R}^N}\omega_{\lambda}^+(x,0) |v_{\varepsilon,z}(x,0)|^{2_\alpha^*-1}dx+o(\varepsilon^{\frac{N-2\alpha}{2}})\notag
\\&\geq C\varepsilon^{\frac{N-2\alpha}{2}}+o(\varepsilon^{\frac{N-2\alpha}{2}})\notag
\\&\geq C\varepsilon^{\frac{N-2\alpha}{2}}
\end{align}
for small $\varepsilon>0$.

Similar to the derivations of (\ref{e3.21}) and (\ref{e3.22}), we can obtain
\begin{equation}\label{e3.27}
 I_\lambda(tv_{\varepsilon,z})<\frac{\alpha}{N}S^\frac{N}{2\alpha}
\end{equation}
for $z\in M$ and small $\varepsilon>0$.

For small $\varepsilon>0$, substituting (\ref{e3.25})-(\ref{e3.27}) into (\ref{e3.24}) yields
\begin{equation}\label{e3.28}
I_\lambda(\omega_\lambda^++tv_{\varepsilon,z})<\alpha_\lambda^++\frac{\alpha}{N}S^\frac{N}{2\alpha}
-C\varepsilon^{\frac{N-2\alpha}{2}},
\end{equation}
where $z\in M$ and $t\in[t_1,t_2]$.

Consequently,  from (\ref{e3.23}) and (\ref{e3.28}) it follows that there exists small $\overline{\varepsilon}_0>0$ such that for $\varepsilon\in(0,\overline{\varepsilon}_0)$ there holds
  \begin{equation*}
\displaystyle\max_{t\geq0}I_\lambda(\omega_\lambda^++tv_{\varepsilon,z})<\alpha_\lambda^+
+\frac{\alpha}{N}S^\frac{N}{2\alpha}-\overline{\sigma}(\overline{\varepsilon}_0)
\end{equation*}
uniformly with respect to $z\in M$, where $\overline{\sigma}(\overline{\varepsilon}_0)$ is a small positive constant.
\end{proof}

\begin{lemma}\label{l3.4}
Assume that $a(x)$ is sign-changing. Then for any $z\in M$ there exists $t_{\varepsilon,z}^->0$ such that
 \begin{equation*}
\omega_\lambda^++t_{\varepsilon,z}^-v_{\varepsilon,z}\in N_\lambda^-.
\end{equation*}
\end{lemma}
\begin{proof}
Define
\begin{equation*}
 U_1:=\left\{\omega\in X^\alpha(\mathbb{R}_+^{N+1})\backslash\{0\};\ \frac{1}{\|\omega\|_X}t^-\left(\frac{\omega}{\|\omega\|_X}\right)>1\right\}\cup\{0\}
\end{equation*}
and
\begin{equation*}
 U_2:=\left\{\omega\in X^\alpha(\mathbb{R}_+^{N+1})\backslash\{0\};\ \frac{1}{\|\omega\|_X}t^-\left(\frac{\omega}{\|\omega\|_X}\right)<1\right\}.
\end{equation*}
Then $N_\lambda^-$ decompose $X^\alpha(\mathbb{R}_+^{N+1})$ into two disjoint connected components $U_1$ and $U_2$.

Since $\omega_\lambda^+\in N_\lambda^+$, we have
$1<t^-(\omega_\lambda^+)$ and $\omega_\lambda^+\in U_1.$

We claim that
\begin{equation*}\label{e3.29}
0<t^-\left(\frac{\omega_\lambda^++tv_{\varepsilon,z}}{\|\omega_\lambda^++tv_{\varepsilon,z}\|_X}
\right)<\overline{C}
\end{equation*}
for some $\overline{C}>0$ and all $t\geq0$. Suppose otherwise that there exists a sequence $\{t_n\}$ such that $t_n\rightarrow\infty$ and
\begin{equation*}
t^-\left(\frac{\omega_\lambda^++t_nv_{\varepsilon,z}}{\|\omega_\lambda^++t_nv_{\varepsilon,z}\|_X}
\right)\rightarrow\infty, \ \text{as}\ n\rightarrow\infty.
\end{equation*}

Set
\begin{equation*}
\upsilon_n=\frac{\omega_\lambda^++t_nv_{\varepsilon,z}}{\|\omega_\lambda^++t_nv_{\varepsilon,z}\|_X}.
\end{equation*}
Recalling $t^-(\upsilon_n)\upsilon_n\in N_\lambda^-\subset N_\lambda$, by Lebesgue's dominated convergence theorem we have
\begin{align*}
&\displaystyle\int_{\mathbb{R}^N} b(x)|\upsilon_n(x,0)|^{2_\alpha^*}dx  \notag
\\  \notag
=& \frac{1}{\|\omega_\lambda^++t_nv_{\varepsilon,z}\|_X^{2_\alpha^*}}\displaystyle\int_{\mathbb{R}^N} b(x)|\omega_\lambda^+(x,0)+t_nv_{\varepsilon,z}(x,0)|^{2_\alpha^*}dx\notag
\\=& \frac{1}{\|\frac{\omega_\lambda^+}{t_n}+v_{\varepsilon,z}\|_X^{2_\alpha^*}}\displaystyle\int_{\mathbb{R}^N} b(x)\left|\frac{\omega_\lambda^+(x,0)}{t_n}+v_{\varepsilon,z}(x,0)\right|^{2_\alpha^*}dx\notag
\\ \rightarrow& \frac{1}{\|v_{\varepsilon,z}\|_X^{2_\alpha^*}}\displaystyle\int_{\mathbb{R}^N} b(x)|v_{\varepsilon,z}(x,0)|^{2_\alpha^*}dx,  \ \text{as}\ n\rightarrow\infty.
\end{align*}
Then
\begin{equation*}
I_\lambda(t^-(\upsilon_n)\upsilon_n)\rightarrow-\infty,  \ \text{as}\ n\rightarrow\infty.
\end{equation*}
This contradicts the fact that $I_\lambda$ is bounded below on $N_\lambda$.

Set
\begin{equation*}
t_\lambda=\frac{\|\omega_\lambda^+\|_X+\sqrt{\overline{C}^2+\|\omega_\lambda^+\|_X^2}}{\|v_{\varepsilon,z}\|_X}+1.
\end{equation*}
By a direct calculation we obtain
\begin{align}
&\|\omega_\lambda^++t_\lambda v_{\varepsilon,z}\|_X^2\notag
\\=&\|\omega_\lambda^+\|_X^2+t^2_\lambda\|v_{\varepsilon,z}\|_X^2+2t_\lambda\left(\displaystyle\int_{\mathbb{R}_+^{N+1}}y^{1-2\alpha}\nabla\omega_\lambda^+\nabla v_{\varepsilon,z} dxdy+\displaystyle\int_{\mathbb{R}^N}\omega_\lambda^+(x,0)v_{\varepsilon,z}(x,0)dx\right)\notag
\\
\geq& \|\omega_\lambda^+\|_X^2+t^2_\lambda\|v_{\varepsilon,z}\|_X^2-2t_\lambda\|v_{\varepsilon,z}\|_X\|\omega_\lambda^+\|_X\notag
\\
>& \left[t^-\left(\frac{\omega_\lambda^++t_\lambda v_{\varepsilon,z}}{\|\omega_\lambda^++t_\lambda v_{\varepsilon,z}\|_X}\right)\right]^2.\notag
\end{align}
That is, $\omega_\lambda^++t_\lambda v_{\varepsilon,z}\in U_2.$
Thus, there exists $0<t_{\varepsilon,z}^-<t_\lambda$ such that
$\omega_\lambda^++t_{\varepsilon,z}^-v_{\varepsilon,z}\in N_\lambda^-$.
\end{proof}

\begin{proof}[Proof of Theorem \ref{t1.2}.]
Let $\omega_n\subset  N_\lambda^-$ be a sequence such that $I_\lambda(\omega_n)\rightarrow \alpha_\lambda^-$
as $n\rightarrow\infty$. It follows from the Ekeland's variational principle \cite{32} that there exists a sequence $\{\widetilde{\omega}_n\}\subset X^\alpha(\mathbb{R}_+^{N+1})$ such that
\begin{equation*}
\widetilde{\omega}_n-\omega_n\rightarrow 0 \  \text{in}~X^\alpha(\mathbb{R}_+^{N+1}),~~I'_\lambda(\widetilde{\omega}_n)\rightarrow0,
~~I_\lambda(\widetilde{\omega}_n)\rightarrow\alpha_\lambda^-, \ \text{as}\ n\rightarrow\infty.
\end{equation*}
According to Lemma \ref{l3.3} (i), we have
\begin{equation}\label{e3.30}
\alpha_\lambda^-<\frac{\alpha}{N}S^\frac{N}{2\alpha}.
\end{equation}
It follows from  (\ref{e3.30}) and Lemma \ref{l3.1} $(ii)$ that there exists $\omega_\lambda^-\in X^\alpha(\mathbb{R}_+^{N+1})$ such that
$\widetilde{\omega}_n\rightarrow \omega_\lambda^-$
as $n\rightarrow\infty$. That is, $\omega_\lambda^-\in N_\lambda~\text{and}~I_\lambda(\omega_\lambda^-)=\alpha_\lambda^-.$

With the help of Lemmas  \ref{l2.3} and  \ref{l2.5} $(ii)$, we find $N_\lambda^+=N_\lambda^0=\emptyset$. Using Lemma \ref{l2.2} indicates that $\omega_\lambda^-$ is a nontrivial  ground state solution of system (\ref{e2.2}), so $\omega_\lambda^-(x,0)$ is a nontrivial ground state solution of equation (\ref{e1.1}).

Note that $I_\lambda(|\omega_\lambda^-|)=\alpha_\lambda^-$. We suppose that $\omega_\lambda^-\geq0$. By virtue of the Maximum Principle for the fractional elliptic equations \cite{13}, we obtain that $\omega_\lambda^-(x,0)$ is a positive ground state solution of equation (\ref{e1.1}).
\end{proof}

\begin{remark}\label{r3.1}
If $a(x)$ is sign-changing, by using Lemmas  \ref{l2.1} $(ii)$,  \ref{l3.1} $(i)$,  \ref{l3.3} $(ii)$ and \ref{l3.4}, and taking closely analogous arguments to the proof of Theorem \ref{t1.2},  we can also obtain that there exists $\widetilde{\omega}_\lambda^-\in N_\lambda^-$ such that $\widetilde{\omega}_\lambda^-(x,0)$ is a positive solution of equation (\ref{e1.1}) and $I_\lambda(\widetilde{\omega}_\lambda^-)=\alpha_\lambda^-$.
\end{remark}

\section{Multiple Positive Solutions}

In this section, we apply the category theory to study multiple positive solutions of equation (\ref{e1.1}) and prove Theorems \ref{t1.3} and \ref{t1.4}. Throughout this section, we always suppose that conditions $(H_1)-(H_3)$ hold.

\begin{proposition}\cite{3}\label{p4.1}
  Let $R$ be a $\mathcal{C}^{1,1}$ complete Riemannian manifold (modelled on a Hilbert space) and assume $F\in \mathcal{C}^1(R,\mathbb{R})$ bounded from below. Let $-\infty<\displaystyle\inf_R F<a<b<+\infty$. Suppose that $F$ satisfies the (PS)-condition on the sublevel $\{u\in R; F(u)\leq b\}$ and $a$ is not a critical level for $F$. Then we have
  \begin{equation*}
  \sharp\{u\in F^a;\ \nabla F(u)=0\}\geq cat_{F^a}(F^a),
  \end{equation*}
  where $F^a\equiv\{u\in R; F(u)\leq a\}$.
  \end{proposition}

\begin{proposition}\cite{3}\label{p4.2}
Suppose that $Q,\, \Omega^+$ and $\Omega^-$ are closed sets with $\Omega^-\subset\Omega^+$. Let $\phi: Q\rightarrow\Omega^+$ and $\varphi: \Omega^-\rightarrow Q$ be two continuous maps such that $\phi\circ\varphi$ is homotopically equivalent to the embedding $j:\Omega^-\rightarrow\Omega^+$. Then $cat_Q(Q)\geq cat_{\Omega^+}(\Omega^-)$.
\end{proposition}


Set
\begin{equation*}
 \dot{X}^\alpha(\mathbb{R}_+^{N+1}):=\left\{\omega(x,y)\in C_0^\infty(\mathbb{R}_+^{N+1});\displaystyle\int_{\mathbb{R}_+^{N+1}}y^{1-2\alpha}|\nabla\omega|^2dxdy<\infty \right\}
  \end{equation*}
equipped with the norm:
\begin{equation*}
 \|\omega\|_{\dot{X}}=\left(\displaystyle\int_{\mathbb{R}_+^{N+1}}y^{1-2\alpha}|\nabla\omega|^2dxdy\right)^\frac{1}{2}.
  \end{equation*}

\begin{lemma}\label{l4.1}
For any $\omega\in X^\alpha(\mathbb{R}_+^{N+1})$, given $\sigma>0$ and $x_0\in\mathbb{R}^N$, we define the following scaled function
\begin{equation*}
\rho(\omega)=\omega_\sigma:(x,y)\mapsto\sigma^\frac{N-\alpha}{2}\omega(\sigma(x-x_0),\sigma y).
\end{equation*}
Then this scaling operation $\rho$ keeps norms $\|\omega_\sigma\|_{\dot{X}^\alpha}$ and $|\omega_\sigma(x,0)|_{2_\alpha^*}$ invariant with respect to $\sigma$, and determined by the ``center'' or ``concentration'' point $x_0$ and the ``modulus'' $\sigma$.
\end{lemma}

\begin{proof}
 We just need to show that $\|\omega_\sigma\|_{\dot{X}^\alpha}=\|\omega\|_{\dot{X}^\alpha}$ and $|\omega_\sigma(x,0)|_{2_\alpha^*}=|\omega(x,0)|_{2_\alpha^*}$. Let $z=\sigma(x-x_0)$ and $t=\sigma y$. We have $dz=\sigma^Ndx$ and $dt=\sigma dy$. Then
 \begin{equation*}
\displaystyle\int_{\mathbb{R}_+^{N+1}}y^{1-2\alpha}|\nabla\omega_\sigma|^2dxdy=\displaystyle\int_{\mathbb{R}_+^{N+1}}t^{1-2\alpha}|\nabla\omega|^2dzdt.
\end{equation*}
Similarly, we can obtain $|\omega_\sigma(x,0)|_{2_\alpha^*}=|\omega(x,0)|_{2_\alpha^*}$.
\end{proof}

Applying Proposition \ref{p2.2} and Lemma \ref{p4.1}, and following \cite[Theorem 2.5]{36}, we can obtain the following result on global compactness immediately.

\begin{lemma}\label{l4.2}
Let $\{\omega_n\}\subset X^\alpha(\mathbb{R}_+^{N+1})\subset \dot{X}^\alpha(\mathbb{R}_+^{N+1})$ be a $(PS)$ sequence for $I_\infty$. Then there exist a number $k\in\mathbb{Z}_+$ and $k$ sequences of points $\{x_n^i\}\subset\mathbb{R}^N\ (1\leq i\leq k)$, and $k+1$ sequences of functions $\{\omega_n^j\}\subset \dot{X}^\alpha(\mathbb{R}_+^{N+1})\ (0\leq j\leq k)$ such that for a sequence, still denoted by $\{\omega_n\}$, we have
\begin{equation*}
\omega_n(x,y)=\omega^0_n(x,y)+\displaystyle\Sigma_{i=1}^k\frac{1}{(\sigma_n^i)^\frac{N-\alpha}{2}}\omega_n^i\left(\frac{x-x_n^i}{\sigma_n^i},\frac{y}{\sigma_n^i}\right)
\end{equation*}
and
\begin{equation*}
\omega_n^j\rightarrow\omega^j~\text{in}~\dot{X}^\alpha(\mathbb{R}_+^{N+1}),~0\leq j\leq k,
\end{equation*}
as $n\rightarrow\infty$, where $\omega^0$ is a solution of
\begin{equation*}
\left\{\begin{array}{ll}
div(y^{1-2\alpha}\nabla\omega)=0,~ ~~~~~~\text{in}~\mathbb{R}_+^{N+1},\\
-\frac{\partial\omega}{\partial\nu}=-\omega+|\omega|^{2_\alpha^*-2}\omega,~~\text{on}~\mathbb{R}^N\times\{0\},
\end{array}\right.
\end{equation*}
$\omega^j \ (1\leq j\leq k)$ are solutions of
\begin{equation*}
\left\{\begin{array}{ll}
div(y^{1-2\alpha}\nabla\omega)=0,~~~\text{in}~\mathbb{R}_+^{N+1},\\
-\frac{\partial\omega}{\partial\nu}=|\omega|^{2_\alpha^*-2}\omega,~~~~~\text{on}~\mathbb{R}^N\times\{0\},
\end{array}\right.
\end{equation*}
and
\begin{itemize}
  \item if $x_n^i\rightarrow\overline{x}^i$, as $n\rightarrow\infty$, then either $\sigma_n^i\rightarrow+\infty$ or $\sigma_n^i\rightarrow0$;
  \item if $|x_n|\rightarrow+\infty$, as $n\rightarrow\infty$, then each of following three cases
\begin{equation*}
\left\{\begin{array}{ll}
\sigma_n^i\rightarrow+\infty,\\
\sigma_n^i\rightarrow0,\\
\sigma_n^i\rightarrow\overline{\sigma}^i,~~0<\overline{\sigma}^i<+\infty
\end{array}\right.
\end{equation*}
can occur.
\end{itemize}

Moreover, we have
\begin{equation*}
\|\omega_n\|_{\dot{X}^\alpha}^2\rightarrow\displaystyle\Sigma_{j=0}^k\|\omega^j\|_{\dot{X}^\alpha}^2
\end{equation*}
and
\begin{equation*}
I_\infty(\omega_n)\rightarrow I_\infty(\omega^0)+\displaystyle\Sigma_{j=1}^kI^\infty(\omega^j)
\end{equation*}
as $n\rightarrow\infty$, where $I^\infty(\omega^j)=\frac{1}{2}\|\omega^j\|_{\dot{X}^\alpha}^2-\frac{1}{2_\alpha^*}\displaystyle\int_{\mathbb{R}^N} |\omega^j(x,0)|^{2_\alpha^*}dx$, $1\leq j\leq k$.
\end{lemma}
The following corollary can be obtained directly from Proposition \ref{p2.2} and Lemma \ref{l4.2}.
\begin{corollary}\label{r4.1}
Let $\{\omega_n\}\subset X^\alpha(\mathbb{R}_+^{N+1})$ be a nonnegative function sequence with $|\omega_n(x,0)|_{2_\alpha^*}=1$ and $\|\omega_n\|_X^2\rightarrow S$. Then there exists a sequence $(x_n,\varepsilon_n)\in\mathbb{R}^N\times\mathbb{R}^+$ such that
\begin{equation*}
\omega_n(x,y):=\frac{1}{S^\frac{N-2\alpha}{4\alpha}}E_\alpha(u_{\varepsilon_n}(x-x_n))+o(1)
\end{equation*}
 in $\dot{X}^\alpha(\mathbb{R}_+^{N+1})$, where $u_\varepsilon$ is defined in Proposition \ref{p2.2}. Moreover, if $x_n\rightarrow\overline{x}$, then $\varepsilon_n\rightarrow0$ or it is unbounded.
\end{corollary}

Define the continuous map $\Phi:X^\alpha(\mathbb{R}_+^{N+1})\backslash G\rightarrow\mathbb{R}^N$ by
\begin{equation*}
\Phi(\omega):=\frac{\displaystyle\int_{\mathbb{R}^N} x|\omega(x,0)-\omega_\lambda^+(x,0)|^{2_\alpha^*}dx}{\displaystyle\int_{\mathbb{R}^N} |\omega(x,0)-\omega_\lambda^+(x,0)|^{2_\alpha^*}dx},
\end{equation*}
where $G=\left\{u\in X^\alpha(\mathbb{R}_+^{N+1});\ \displaystyle\int_{\mathbb{R}^N} |\omega(x,0)-\omega_\lambda^+(x,0)|^{2_\alpha^*}dx=0 \right\}$, and define another map $\widehat{\Phi}:X^\alpha(\mathbb{R}_+^{N+1})\backslash \{0\}\rightarrow\mathbb{R}^N$ by
\begin{equation*}
\widehat{\Phi}(\omega):=\frac{\displaystyle\int_{\mathbb{R}^N} x|\omega(x,0)|^{2_\alpha^*}dx}{\displaystyle\int_{\mathbb{R}^N} |\omega(x,0)|^{2_\alpha^*}dx}.
\end{equation*}

\begin{lemma}\label{l4.3}
\begin{description}
  \item[(i)] For each $0<\delta<r_0$, there exist $\lambda_\delta,\delta_0>0$ such that if $\omega\in N_\infty$ with $I_\infty(\omega)<\frac{\alpha}{N}S^\frac{N}{2\alpha}+\delta_0$ and $\lambda\in(0,\lambda_\delta)$, then $\Phi(\omega)\in M_\delta$.
  \item[(ii)] For each $0<\delta<r_0$, there exists $\overline{\delta}_0>0$ such that if $\omega\in N_\infty$ with $I_\infty(\omega)<\frac{\alpha}{N}S^\frac{N}{2\alpha}+\overline{\delta}_0$, then $\widehat{\Phi}(\omega)\in M_\delta$.
\end{description}
\end{lemma}

\begin{proof}
$\textbf{(i)}$. Suppose on the contrary that there exists a sequence $\{\omega_n\}\subset N_\infty$ such that $I_\infty(\omega_n)<\frac{\alpha}{N}S^\frac{N}{2\alpha}+o_n(1)$,
$\lambda\rightarrow0^+$, and
\begin{equation}\label{e4.2}
\Phi(\omega_n)\not\in M_\delta ~\text{for~all}~n\in\mathbb{Z}_+.
\end{equation}
Since
 \begin{align}
\frac{\alpha}{N}S^\frac{N}{2\alpha}+1&>I_\infty(\omega_n)-\frac{1}{2_\alpha^*}I'_\infty(\omega_n)\omega_n\notag
\\&=\left(\frac{1}{2}-\frac{1}{2_\alpha^*}\right)\|\omega_n\|_X^2,\notag
\end{align}
we obtain that $\{\omega_n\}$ is bounded in $X^\alpha(\mathbb{R}_+^{N+1})$.

From (\ref{ss3.16}) and Lemma \ref{l3.20}, there is a sequence $\{t_n\}\subset\mathbb{R}^+$:
\begin{equation*}
t_n:=\left(\frac{\|\omega_n\|_X^2}{\displaystyle\int_{\mathbb{R}^N}  |\omega_n(x,0)|^{2_\alpha^*} dx}\right)^\frac{1}{2_\alpha^*-2}
\end{equation*}
such that $\{t_n \omega_n\}\in N^\infty$ and $
\frac{\alpha}{N}S^\frac{N}{2\alpha}\leq I^\infty(t_n\omega_n)
\leq I_\infty(t_n\omega_n)\leq I_\infty(\omega_n)=\frac{\alpha}{N}S^\frac{N}{2\alpha}+o_n(1).$
Thus, we have $t_n=1+o_n(1)$  and
\begin{align}\label{e4.3}
\displaystyle\lim_{n\rightarrow\infty}I_\infty(\omega_n)&=\displaystyle\lim_{n\rightarrow\infty}\frac{\alpha}{N}\|\omega_n\|_X^2
\nonumber
\\&=\displaystyle\lim_{n\rightarrow\infty}\frac{\alpha}{N}\displaystyle\int_{\mathbb{R}^N}b(x)|\omega_n(x,0)|^{2_\alpha^*} dx
\nonumber\\&=\frac{\alpha}{N}S^\frac{N}{2\alpha}+o_n(1).
\end{align}

Set
\begin{equation*}
U_n=\frac{\omega_n}{\left(\displaystyle\int_{\mathbb{R}^N}|\omega_n(x,0)|^{2_\alpha^*} dx\right)^\frac{1}{2_\alpha^*}}.
\end{equation*}
Then $\displaystyle\int_{\mathbb{R}^N}|U_n(x,0)|^{2_\alpha^*} dx=1$. It follows from (\ref{e4.3}) that
$\displaystyle\lim_{n\rightarrow\infty}\|U_n\|_X^2=S.$
According to  Corollary \ref{r4.1},   there exists a sequence $\{(x_n,\varepsilon_n)\}\subset\mathbb{R}^N\times\mathbb{R}^+$ such that
\begin{equation}\label{e4.4}
U_n(x,y):=\frac{1}{S^\frac{N-2\alpha}{4\alpha}}E_\alpha(u_{\varepsilon_n}(x-x_n))+o_n(1).
\end{equation}
Moreover, if $\{x_n\}\rightarrow\overline{x}$, then $\varepsilon_n\rightarrow0$ or it is unbounded.


$\textbf{Case 1}$. Suppose that $\{x_n\}\rightarrow\infty$ as $n\rightarrow\infty$. Without loss of generality, we assume that $b(x_n)\rightarrow b_\infty$ as $n\rightarrow\infty$, where $b_\infty$ is defined in Remark \ref{r1.2}. From (\ref{e4.3}) and (\ref{e4.4}) we deduce that
\begin{align}\label{e4.5}
1&=\frac{\displaystyle\int_{\mathbb{R}^N}b(x)|\omega_n(x,0)|^{2_\alpha^*} dx}{\displaystyle\int_{\mathbb{R}^N}|\omega_n(x,0)|^{2_\alpha^*} dx}+o_n(1)\notag
\\&=\displaystyle\int_{\mathbb{R}^N}b(x)|U_n(x,0)|^{2_\alpha^*} dx+o_n(1)\notag
\\&=S^{-\frac{N}{2\alpha}}\displaystyle\int_{\mathbb{R}^N}b(x+x_n)|u_{\varepsilon_n}(x)|^{2_\alpha^*} dx+o_n(1)\notag
\\&\leq b_\infty.\notag
\end{align}
This contradicts the fact of $b_\infty<1$.

$\textbf{Case 2}$. Suppose that $\{x_n\}\rightarrow\overline{x}$ as $n\rightarrow\infty$.  Using Corollary \ref{r4.1}, we have $\varepsilon_n\rightarrow0$ as $n\rightarrow\infty$. It follows from (\ref{e4.3})-(\ref{e4.4}) that
\begin{align}
1&=\frac{\displaystyle\int_{\mathbb{R}^N}b(x)|\omega_n(x,0)|^{2_\alpha^*} dx}{\displaystyle\int_{\mathbb{R}^N}|\omega_n(x,0)|^{2_\alpha^*} dx}+o_n(1)\notag
\\&=\displaystyle\int_{\mathbb{R}^N}b(x)|U_n(x,0)|^{2_\alpha^*} dx+o_n(1)\notag
\\&=S^{-\frac{N}{2\alpha}}\displaystyle\int_{\mathbb{R}^N}b(\sqrt{\varepsilon_n}x+x_n)|u_1(x)|^{2_\alpha^*} dx+o_n(1)\notag
\\&=b(\overline{x}),
\end{align}
where $u_1(x)=u_\varepsilon(x)$ when $\varepsilon=1$. In view of (\ref{e4.5}) and $\overline{x}\in M$, we have
\begin{align*}
\Phi(\omega_n)&=\frac{\displaystyle\int_{\mathbb{R}^N} x|\omega_n(x,0)-\omega_\lambda^+(x,0)|^{2_\alpha^*}dx}{\displaystyle\int_{\mathbb{R}^N} |\omega_n(x,0)-\omega_\lambda^+(x,0)|^{2_\alpha^*}dx}\notag
\\&=\frac{\displaystyle\int_{\mathbb{R}^N} x|\omega_n(x,0)|^{2_\alpha^*}dx}{\displaystyle\int_{\mathbb{R}^N} |\omega_n(x,0)|^{2_\alpha^*}dx}+o_\lambda(1),~\text{as}~\lambda\rightarrow0\notag
\\
&=\frac{\displaystyle\int_{\mathbb{R}^N} (x_n+\sqrt{\varepsilon_n}x)|u_1(x)|^{2_\alpha^*}dx}{\displaystyle\int_{\mathbb{R}^N} |u_1(x)|^{2_\alpha^*}dx}+o_\lambda(1)\notag
\\&\rightarrow \overline{x}\in M,\ \text{as}~n\rightarrow\infty.\notag
\end{align*}

This yields a contradiction with (\ref{e4.2}). That is, Part $(i)$ holds.

Processing in an analogous manner, we can arrive at Part $(ii)$, so we omit it.
\end{proof}

\begin{lemma}\label{l4.4}
\begin{description}
  \item[(i)]There exists $\Lambda_\delta>0$ small enough such that if $\lambda\in(0,\Lambda_\delta)$ and $\omega\in N_\lambda^-$ with $I_\lambda(\omega)<\frac{\alpha}{N}S^\frac{N}{2\alpha}+\frac{\delta_0}{2}$ ($\delta_0$ is given in Lemma \ref{l4.3} $(i)$), then $\Phi(\omega)\in M_\delta$.
  \item[(ii)] There exists $\overline{\Lambda}_\delta>0$ small enough such that if $\lambda\in(0,\overline{\Lambda}_\delta)$ and $\omega\in N_\lambda^-$ with $I_\lambda(\omega)<\frac{\alpha}{N}S^\frac{N}{2\alpha}+\frac{\overline{\delta}_0}{2}$ ($\overline{\delta}_0$ is given in Lemma \ref{l4.3} $(ii)$), then $\widehat{\Phi}(\omega)\in M_\delta$.
\end{description}
\end{lemma}

\begin{proof}
We only prove Part $(i)$. The proof of Part $(ii)$ is closely similar.

For any $\omega\in N_\lambda^-$ with $I_\lambda(\omega)<\frac{\alpha}{N}S^\frac{N}{2\alpha}+\frac{\delta_0}{2}$, there is a unique number:
\begin{equation*}
t_\infty(\omega)=\left(\frac{\|\omega\|_X^2}{\displaystyle\int_{\mathbb{R}^N}b(x)|\omega(x,0)|^{2_\alpha^*} dx}\right)^\frac{N-2\alpha}{4\alpha}
\end{equation*}
such that $t_\infty(\omega)\omega\in N_\infty$. We now claim that there are some $C_1, C_2>0$ independent of $\omega$ such that
\begin{equation}\label{e4.6}
C_1\leq t_\infty(\omega)\leq C_2.
\end{equation}
In view of $I_\lambda(\omega)<\frac{\alpha}{N}S^\frac{N}{2\alpha}+\frac{\delta_0}{2}$ and $I'_\lambda(\omega)\omega=0$, as we discussed for (\ref{e3.1}), there holds
\begin{equation}\label{e4.7}
\|\omega\|_X\leq C,
\end{equation}
where $C$ is independent of $\omega$.

On the other hand, it follows from (\ref{e2.6}) and Proposition \ref{p2.1} that
  \begin{align}\label{e4.8}
\displaystyle\int_{\mathbb{R}^N}a(x)|\omega(x,0)|^2\ln|\omega(x,0)|dx&\leq C(\|\omega\|_X^2+|\omega(x,0)|_2^4)\leq C(\|\omega\|_X^2+\|\omega\|_X^4).
\end{align}
From (\ref{e2.8}), (\ref{e4.8}) and Proposition \ref{p2.1}, we deduce that
\begin{align}
\|\omega\|_X^2&=\lambda\displaystyle\int_{\mathbb{R}^N}a(x)|\omega(x,0)|^2\ln|\omega(x,0)|dx+\displaystyle\int_{\mathbb{R}^N}b(x)|\omega_n(x,0)|^{2_\alpha^*}dx\notag
\\&\leq\lambda C(\|\omega\|_X^2+\|\omega\|_X^4)+C\|\omega\|_X^{2_\alpha^*}.\notag
\end{align}
That is,
\begin{equation*}
 \|\omega\|_X^{2_\alpha^*}+\|\omega\|_X^4\geq (1-\lambda C)\|\omega\|_X^2\geq C\|\omega\|_X^2
 \end{equation*}
for small $\lambda>0$ and some $C>0$. Hence, we have
\begin{equation}\label{e4.9}
 \|\omega\|_X^2\geq C_3
 \end{equation}
for some $C_3>0$ independent of $\omega$.

Taking into account Proposition \ref{p2.2} with (\ref{e2.8}) and (\ref{e4.7})-(\ref{e4.9}), we have
\begin{equation}\label{e4.10}
\displaystyle\int_{\mathbb{R}^N}b(x)|\omega(x,0)|^{2_\alpha^*}dx\leq C\|\omega\|_X^{2_\alpha^*}\leq C_4
 \end{equation}
and
\begin{align}\label{e4.11}
\displaystyle\int_{\mathbb{R}^N}b(x)|\omega(x,0)|^{2_\alpha^*}dx&=\|\omega\|_X^2-\lambda\displaystyle\int_{\mathbb{R}^N}a(x)|\omega(x,0)|^2\ln|\omega(x,0)|dx\notag
\\&\geq \|\omega\|_X^2-\lambda C(\|\omega\|_X^2+\|\omega\|_X^4)\notag
\\&\geq \|\omega\|_X^2(1-\lambda C)\notag
\\&\geq C_5
\end{align}
for small $\lambda>0$ and some positive constants $C_4, C_5>0$.
Clearly, by combining (\ref{e4.7}) and (\ref{e4.9})- (\ref{e4.11}), we see that (\ref{e4.6}) holds.

In view of (\ref{e4.6})-(\ref{e4.8}) and Proposition \ref{p2.1}, we obtain
\begin{align}
I_\infty(t_\infty(\omega)\omega)&=I_\lambda(t_\infty(\omega)\omega)+\frac{\lambda}{2}\displaystyle\int_{\mathbb{R}^N}a(x)|t_\infty(\omega)\omega(x,0)|^2\ln|t_\infty(\omega)\omega(x,0)|dx\notag
\\&~~~~-\frac{\lambda}{4}\displaystyle\int_{\mathbb{R}^N}a(x)|t_\infty(\omega)\omega(x,0)|^2dx\notag
\\&\leq\displaystyle\max_{t\geq0}I_\lambda(t\omega)+\lambda C\notag
\\&\leq\frac{\alpha}{N}S^\frac{N}{2\alpha}+\frac{\delta_0}{2}+\lambda C\notag
\\&<\frac{\alpha}{N}S^\frac{N}{2\alpha}+\delta_0\notag
\end{align}
for small $\lambda>0$.
\end{proof}

Set
\begin{displaymath}
c_\lambda:=\left\{\begin{array}{ll}
\frac{\alpha}{N}S^\frac{N}{2\alpha}-\sigma(\varepsilon_0),~ &\text{if}~a(x)\geq0,\\
\alpha_\lambda^++\frac{\alpha}{N}S^\frac{N}{2\alpha}-\overline{\sigma}(\overline{\varepsilon}_0),~ &\text{if}~a(x)~\text{is~sign-changing},
\end{array}\right.
\end{displaymath}
 and
\begin{equation*}
N_\lambda^-(c_\lambda):=\{\omega\in N_\lambda^-;I_\lambda(\omega)\leq c_\lambda\},
\end{equation*}
where $\sigma(\varepsilon_0)$ and $\overline{\sigma}(\overline{\varepsilon}_0)$ are defined in Lemma \ref{l3.3}.

Denote by $I_{N_\lambda^-}$ the restriction of $I_\lambda$ on $N_\lambda^-$. Then we have

\begin{lemma}\label{l4.5}
$I_{N_\lambda^-}$ satisfies the $(PS)$-condition on $N_\lambda^-(c_\lambda)$.
\end{lemma}
\begin{proof}
Let $\{\omega_n\}\subset N_\lambda^-(c_\lambda)$ be a $(PS)$ sequence.
Then there exists a sequence $\{\theta_n\}\subset\mathbb{R}$ such that
\begin{equation*}
 I'_\lambda(\omega_n)=\theta_n\Psi'_\lambda(\omega_n)+o_n(1),
\end{equation*}
where $\Psi_\lambda$ is defined in (\ref{ee2.9}).
Since $\omega_n\in N_\lambda^-$, we have $\Psi'_\lambda(\omega_n)\omega_n<0$ and there exists a subsequence (still denoted by $\{\omega_n\}$) such that
$\Psi'_\lambda(\omega_n)\omega_n\rightarrow l\leq0,~n\rightarrow\infty.$

If $l=0$, then
\begin{equation*}
\|\omega_n\|_X^2-\lambda\displaystyle\int_{\mathbb{R}^N}a(x)|\omega_n(x,0)|^2\ln|\omega_n(x,0)|dx-\displaystyle\int_{\mathbb{R}^N}b(x)|\omega_n(x,0)|^{2_\alpha^*}dx=o_n(1)
\end{equation*}
and
\begin{align}
 &\|\omega_n\|^2_X-\lambda\displaystyle\int_{\mathbb{R}^N}a(x)|\omega_n(x,0)|^2\ln|\omega_n(x,0)|dx\notag
-\lambda\displaystyle\int_{\mathbb{R}^N}a(x)|\omega_n(x,0)|^2dx \\& -(2_\alpha^*-1)\displaystyle\int_{\mathbb{R}^N}b(x)|\omega_n(x,0)|^{2_\alpha^*}dx=o_n(1).\notag
\end{align}
The rest is similar to the proof of Lemma \ref{l2.3},
we can obtain a contradiction. Hence, $l<0$. Due to $I'_\lambda(\omega_n)\omega_n=0$, we conclude that $\{\theta_n\}\rightarrow0$ as $n\rightarrow\infty$. Consequently, we obtain
$I'_\lambda(\omega_n)\rightarrow0,~n\rightarrow\infty.$ By virtue of Lemma \ref{l3.1}, we arrive at the desired result.
\end{proof}

\begin{proof}[Proof of Theorem \ref{t1.3}]
Let $\delta,\Lambda_\delta>0$ be the same as given in Lemmas \ref{l4.3} $(i)$ and \ref{l4.4} $(i)$ respectively. Firstly, we show that  $I_\lambda$ has at least $cat_{M_\delta}(M)$ critical points in $N_\lambda^-(c_\lambda)$ for $\lambda\in(0,\Lambda_\delta)$. For $z\in M$, by Lemmas \ref{l3.3} $(ii)$ and  \ref{l3.4}, we define
 \begin{equation*}
F(z)=\omega_\lambda^++t_{\varepsilon,z}^-v_{\varepsilon,z},
\end{equation*}
which belongs to $ N_\lambda^-(c_\lambda)$. Note that  $\Phi(N_\lambda^-(c_\lambda))\subset M_\delta$ for $\lambda<\Lambda_\delta$ by Lemma \ref{l4.4} $(i)$.

Define $\xi:[0,1]\times M\rightarrow M_\delta$ by
\begin{equation*}
\xi(\theta,z)=\Phi\left(\omega_\lambda^++t_{(1-\theta)\varepsilon,z}^-v_{(1-\theta)\varepsilon,z}\right)\in N_\lambda^-(c_\lambda).
\end{equation*}
By a straightforward calculation we have $\xi(0,z)=\Phi\circ F(z)$ and $\displaystyle\lim_{\theta\rightarrow1^-}\xi(\theta,z)=z$.
Hence, $\Phi\circ F$ is homotopic to the inclusion $j:M\rightarrow M_\delta$. Combining Lemma \ref{l4.5} with Propositions \ref{p4.1} and \ref{p4.2}, we obtain that $I_{N_\lambda^-(c_\lambda)}$ has at least $cat_{M_\delta}(M)$ critical points in $N_\lambda^-(c_\lambda)$. In addition, from Lemma \ref{l2.2}, $I_\lambda$ has at least $cat_{M_\delta}(M)$ critical points  in $N_\lambda^-(c_\lambda)$.  According to the fact $N_\lambda^+\cap N_\lambda^-=\emptyset$ and Theorem \ref{t1.1},  $I_\lambda$ also has at least $cat_{M_\delta}(M)+1$ critical points.

To show that problem (\ref{e1.1}) admits at least $cat_{M_\delta}(M)+1$ positive solutions, we set
 \begin{align}
I^+_\lambda(\omega)=& \frac{1}{2}\|\omega\|_X^2-\frac{\lambda}{2}\displaystyle\int_{\mathbb{R}^N}a(x)|\omega^+(x,0)|^2\ln|\omega^+(x,0)|dx\notag
\\&+\frac{\lambda}{4}\displaystyle\int_{\mathbb{R}^N}a(x)|\omega^+(x,0)|^2dx-\frac{1}{2_\alpha^*}\displaystyle\int_{\mathbb{R}^N}b(x)|\omega^+(x,0)|^{2_\alpha^*}dx,\notag
\end{align}
where $\omega^+(x,y):=\max\{\omega(x,y),0\}$. Processing in an analogous manner, we can prove that $I^+_\lambda$ has  at least $cat_{M_\delta}(M)+1$ critical points. Suppose that $\omega$ is one of the critical pints of $I^+_\lambda$. Note that
 \begin{equation*}
(I^+_\lambda)'(\omega)\omega^-(x,y)=\|\omega^-\|_X^2=0,
\end{equation*}
 with $\omega^-(x,y)=\min\{\omega(x,y),0\}$. We have $\omega\geq0$.
 By virtue of the maximum Principle for the fractional elliptic equations \cite{13}, we obtain that the problem (\ref{e2.2}) admits at least $cat_{M_\delta}(M)+1$ positive solutions.
\end{proof}

\begin{proof}[Proof of Theorem \ref{t1.4}]
Let $\delta,\overline{\Lambda}_\delta>0$ be the same as given  Lemmas  \ref{l4.3} $(ii)$ and  \ref{l4.4} $(ii)$ respectively. Firstly, we show that  $I_\lambda$ has at least $cat_{M_\delta}(M)$ critical points in $N_\lambda^-(c_\lambda)$ for $\lambda\in(0,\overline{\Lambda}_\delta)$. For $z\in M$, by Lemma \ref{l3.3} $(i)$, we define
 \begin{equation*}
F(z)=t_{\varepsilon,z}^-v_{\varepsilon,z},
\end{equation*}which belongs to $N_\lambda^-(c_\lambda)$.
It follows from Lemma \ref{l4.4} that $\widehat{\Phi}(N_\lambda^-(c_\lambda))\subset M_\delta$ for $\lambda<\overline{\Lambda}_\delta$.

Define $\xi:[0,1]\times M\rightarrow M_\delta$ by
\begin{equation*}
\xi(\theta,z)=\widehat{\Phi}\left(t_{(1-\theta)\varepsilon,z}^-v_{(1-\theta)\varepsilon,z}\right)\in N_\lambda^-(c_\lambda).
\end{equation*}
Then, $\xi(0,z)=\widehat{\Phi}\circ F(z)$ and $\displaystyle\lim_{\theta\rightarrow1^-}\xi(\theta,z)=z$.
Hence, $\widehat{\Phi}\circ F$ is homotopic to the inclusion $j:M\rightarrow M_\delta$. Combining  Lemma \ref{l4.5} with Propositions \ref{p4.1} and \ref{p4.2}, we obtain that $I_{N_\lambda^-(c_\lambda)}$ has at least $cat_{M_\delta}(M)$ critical points in $N_\lambda^-(c_\lambda)$. According to Lemma \ref{l2.2}, we know that $I_\lambda$ has at least $cat_{M_\delta}(M)$ critical points  in $N_\lambda^-(c_\lambda)$.

Similar to the proof of Theorem \ref{t1.3}, we can show that the problem (\ref{e2.2}) admits at least $cat_{M_\delta}(M)$ positive solutions.
\end{proof}

\bibliographystyle{amsplain}

\end{document}